\documentclass[11pt]{article}
\usepackage{hyperref}
\usepackage{a4,fullpage,color}
\usepackage{amssymb,times}
\usepackage{graphicx}
\usepackage{amsmath}
\usepackage{amsfonts}
\usepackage{enumerate}
\usepackage{latexsym}
\usepackage{epsfig}
\usepackage{amsthm}

\makeatletter\@addtoreset{equation}{section}\makeatother
\makeatletter\@addtoreset{figure}{section}\makeatother
\makeatletter\@addtoreset{table}{section}\makeatother

\newtheorem{theorem}{Theorem}[section]
\newtheorem{lemma}[theorem]{Lemma}
\newtheorem{claim}[theorem]{Claim}

\newtheorem{prop}[theorem]{Proposition}

\newcommand{\R}{{\mathbb R}}

\newcommand{\Z}{{\mathbb Z}}

\newcommand{\N}{{\mathbb N}}
\newcommand{\T}{{\mathbb T}}
\newcommand{\D}{{\mathbb D}}

\newcommand{\phy}{\varphi}
\newcommand{\trsp}{\raisebox{.6ex}{${\scriptstyle t}$}}
\newcommand{\op}[1]{\!\!\mathop{\rm ~#1}\nolimits}
\newcommand{\DD}{\!\mathop{\rm d\!}\nolimits}

\newcommand{\scriptop}[1]{\!\!\mathop{\mbox{\rm \scriptsize ~#1}}\nolimits}

\newcommand{\deriv}[2]{\frac{\partial #1}{\partial #2}}

\newcommand{\ham}[1]{\mathcal{X}_{#1}}

\newcommand{\demon}{Proof}
\newcommand{\ssi}{\Longleftrightarrow}

\newcommand{\fleche}{\rightarrow}

\newcommand{\cqfd}{\hfill $\square$\par\vspace{1ex}}

\newcommand{\ouf}{\vspace{3mm}}

\newcommand{\RM}{\mathbb{R}}
\newcommand{\ZM}{\mathbb{Z}}

\newcommand{\NM}{\mathbb{N}}

\renewcommand{\geq}{\geqslant}

\newcommand{\demons}[1][$\!\!$]{\noindent\textbf{\demon\ }\textsl{#1}\textbf{.}~}

\newenvironment{demo}[1][$\!\!$]
{\demons[#1]\ }
{\cqfd}

\newenvironment{remark}{\par\medskip\noindent{\bf
Remark~\thetheorem~~}}{\unskip\nobreak\hfill\hbox{ $\oslash$}\par\bigskip}

\newenvironment{definition}{\refstepcounter{theorem}\par\medskip\noindent{\bf
Definition~\thetheorem~~}}{\unskip\nobreak\hfill\hbox{ $\oslash$}\par\bigskip}

\renewcommand{\emptyset}{\varnothing}

\title{Constructing integrable systems of semitoric type
} \author{Alvaro Pelayo\thanks{Partially supported by an NSF
    Postdoctoral Fellowship.} \,\, and San V\~u Ng\d oc} \date{}
\begin{document}
\maketitle

\begin{abstract}
  Let $(M,\, \omega)$ be a connected, symplectic $4$\--manifold.  A
  \emph{semitoric integrable system on $(M,\, \omega)$} essentially
  consists of a pair of independent, real\--valued, smooth functions
  $J$ and $H$ on the manifold $M$, for which $J$ generates a
  Hamiltonian circle action under which $H$ is invariant.
  In this paper we give a general method to construct, starting from a
  collection of five ingredients, a symplectic $4$\--manifold 
  equipped a semitoric integrable system. Then we show that
  every semitoric integrable
  system on a symplectic $4$\--manifold is obtained in
  this fashion.
  In conjunction with the uniqueness theorem proved recently by
  the authors (Invent. Math. 2009), this gives a classification of semitoric
  integrable systems on $4$\--manifolds, in terms of five invariants.
  Some of the invariants are geometric, others are analytic and others
  are combinatorial/group\--theoretic. 
\end{abstract}

\section{Introduction}

The present paper is motivated by some remarkable results proven in
the 80s by Atiyah, Guillemin\--Sternberg and Delzant, in the context
of Hamiltonian torus actions. Indeed, Atiyah \cite[Th.~1]{atiyah} and
Guillemin\--Sternberg \cite{gs} proved that if an $n$\--dimensional
torus acts on a compact, connected symplectic manifold $(M,\, \omega)$
in a Hamiltonian fashion, the image $\mu(M)$ under the momentum map $
\mu:=(\mu_1,\ldots,\, \mu_n) \colon M \to \R^n $ is a convex polytope.
Delzant \cite{delzant} showed that if the dimension $n$ of the torus
is half the dimension of $M$, this polytope, which in this case is
called a {\em Delzant polytope} (i.e. a convex polytope with the
property that at each vertex of it there are precisely $n$ codimension
one faces with normals which form a $\Z$\--basis of the integral
lattice $\Z^n$) determines the isomorphism type of $M$, and moreover,
$M$ is a toric variety. He also showed that starting from any Delzant
polytope one can construct a symplectic manifold with a Hamiltonian
torus action for which its associated polytope is the one we started
with.

From the viewpoint of symplectic geometry, the situation described by
the momentum polytope is, nevertheless, very rigid.  It is natural to
wonder whether any of these striking results persist in the case where
the torus is replaced by a non\--compact group acting
Hamiltonianly. The seemingly symplest case happens when the group is
$\R^n$, and the study of these $\R^n$\--actions is precisely the goal
of the theory of integrable systems.
Building on previous work of the authors, and of many other authors,
we shall present a ``Delzant'' type classification for integrable
systems, for which one component of the system is generated by a
Hamiltonian circle action; these systems are called semitoric.

Let $(M, \, \omega)$ be a connected, symplectic $4$\--dimensional
manifold, where we do not assume that $M$ is compact.  Any smooth
function $f$ on $M$ induces a unique vector field $\ham{f}$ on $M$
which satisfies $\omega(\ham{f},\, \cdot)=-\op{d}\!f$. It is called
the \emph{Hamiltonian vector field induced by $f$}.  An
\emph{integrable system on $M$} is a pair of real\--valued smooth
functions $J$ and $H$ on $M$, for which the Poisson bracket
$\{J,\,H\}:=\omega(\ham{J},\, \ham{H})$ identically vanishes on $M$,
and the differentials $\op{d}\!J$, $\op{d}\!H$ are almost\--everywhere
linearly independendent.  Of course, here $(J,\,H) \colon M \to \R^2$
is the analogue of the momentum map in the case of a torus action.  In
some local symplectic coordinates of $M$, $(x,\, y,\, \xi,\, \eta)$,
the symplectic form $\omega$ is given by $\op{d}\!  \xi \wedge
\op{d}\!x +\op{d}\! \eta\wedge \op{d}\!y$, and the vanishing of the
Poisson brackets $\{J,\,H\}$ amounts to the partial differential
equation
  $$
  \frac{\partial J}{\partial \xi} \, \frac{\partial H}{\partial x} -
  \frac{\partial J}{\partial x} \, \frac{\partial H}{\partial \xi} +
  \frac{\partial J}{\partial \eta} \, \frac{\partial H}{\partial y} -
  \frac{\partial J}{\partial y} \, \frac{\partial H}{\partial \eta}
  =0.
  $$
  This condition is equivalent to $J$ being constant along the
  integral curves of $\ham{H}$ (or $H$ being constant along the
  integral curves of $\ham{J}$).

  A {\em semitoric integrable system} on $M$ is an integrable system
  for which the component $J$ is a proper momentum map for a
  Hamiltonian circle action on $M$, and the associated map
  $F:=(J,\,H):M\to\R^2$ has only non\--degenerate singularities in the
  sense of Williamson, without real\--hyperbolic blocks.  We also use
  the term {\em $4$\--dimensional semitoric integrable system} to
  refer to the triple $(M,\, \omega,\, (J,\,H))$.  Recall that the
  properness of $J$ means that the preimage by $J$ of a compact set is
  compact in $M$ (which is immediate if $M$ is compact), and the
  non\--degeneracy hypothesis for $F$ means that, if $p$ is a critical
  point of $F$, then there exists a 2 by 2 matrix $B$ such that, if we
  denote $\tilde{F}=B\circ F,$ one of the following situations holds
  in some local symplectic coordinates near $p$~:
  \begin{itemize}
  \item[(1)] $\tilde{F}(x,\,y,\,\xi,\,\eta)=(\eta +
    \mathcal{O}(\eta^2),\,x^2+\xi^2 + \mathcal{O}((x,\,\xi)^3))$
  \item[(2)]
    $\op{d}^2_m\!\tilde{F}(x,\,y,\,\xi,\,\eta)=(x^2+\xi^2,\,y^2+\eta^2)$
  \item[(3)]
    $\op{d}^2_m\!\tilde{F}(x,\,y,\,\xi,\,\eta)=(x\xi+y\eta,\,x\eta-y\xi)$
  \end{itemize}
  The first case is called a transversally --- or codimension 1 ---
  {\em elliptic singularity}; the second case is an {\em
    elliptic\--elliptic singularity}; the last case is a {\em
    focus\--focus singularity}.  In \cite[Th.~6.2]{pelayovungoc} the
  authors constructed, starting from a given semitoric integrable
  system on a $4$\--manifold, a collection of five symplectic
  invariants associated with it and proved that these completely
  determine the integrable system up to isomorphisms.  The goal of the
  present is to complement that work, by providing a general method to
  construct \emph{any} $4$\--dimensional semitoric integrable system
  starting from an abstract collection of ingredients.
Both throughout \cite{pelayovungoc} and the present paper we
make a generic assumption on our semitoric systems; this is
explained in Section \ref{taylor:sec}.

  The symplectic invariants constructed in \cite{pelayovungoc}, for a
  given $4$\--dimensional semitoric integrable system, are the
  following: (i) \emph{the number of singularities invariant}: an
  integer $m_f$ counting the number of isolated singularities; (ii)
  \emph{the singularity type invariant}: a collection of $m_f$
  infinite Taylor series on two variables which classifies locally the
  type of singularity; (iii) \emph{the polygon invariant}: the
  equivalence class of a weighted rational
  convex\footnote{generalizing the Delzant polygon and which may be
    viewed as a bifurcation diagram} polygon $$\Big(\Delta,\,
  (\ell_j)_{j=1}^{m_f},\, (\epsilon_j)_{j=1}^{m_f}\Big).$$ Here
  $\Delta$ is a convex polygon in $\R^2$, the $\ell_j$ are vertical
  lines intersecting $\Delta$ and the $\epsilon_j$ are $\pm 1$ signs
  giving each line $\ell_j$ an orientation;
  \begin{figure}[htb]
    \begin{center}
      \includegraphics{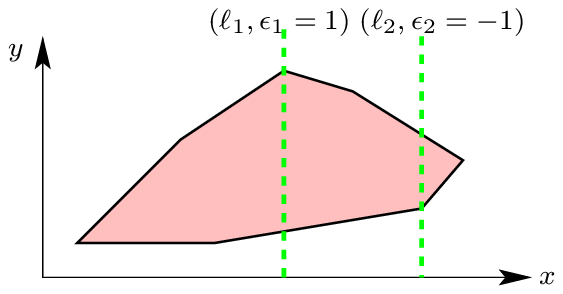}
      \caption{Weighted polygon $(\Delta, (\ell_1,\, \ell_2),\,
        (1,\,-1))$.}
    \end{center}
    \label{AF3}
  \end{figure}
  (iv) \emph{the volume invariant}: $m_f$ numbers measuring volumes of
  certain submanifolds at the singularities; (v) \emph{the twisting
    index invariant}: $m_f$ integers measuring how twisted the system
  is around singularities. This is a subtle invariant, which depends
  on the representative chosen in (iii).  Here, we write $m_f$ to
  emphasize that the singularities that $m_f$ counts are focus\--focus
  singularities.  We then proved that two semitoric systems $(M,\,
  \omega_1,\,(J_1,\,H_1))$ and $(M,\, \omega_2,\, (J_2,\,H_2))$ are
  isomorphic if and only if they have the same invariants (i)--(v),
  where an isomorphism is a symplectomorphism $ \varphi \colon M_1 \to
  M_2 $ such that $ \varphi^*(J_2,\,H_2)=(J_1,\,f(J_1,\,H_1)) $ for
  some smooth function $f$.

  We have found that some restrictions on these symplectic invariants
  must be imposed.  Indeed, we call ``semitoric list of ingredients''
  the following collection of items (i)\--(v): (i) any integer number
  $0 \le m_f<\infty$; (ii) an $m_{f}$\--tuple of real formal power
  series in two variables, with vanishing constant term and first
  terms $\sigma_1\,X+\sigma_2\,Y$ with $\sigma_2 \in [0,\,2\, \pi)$;
  (iii) a Delzant weighted polygon $\Big(\Delta,\,
  (\ell_j)_{j=1}^{m_f},\, (\epsilon_j)_{j=1}^{m_f}\Big)$, of
  complexity $m_f$, where $\Delta$ is a polygon, the $\ell_j$ are
  again vertical lines intersecting $\Delta$ and the $\epsilon_j$ are
  $\pm1$ signs giving each line $\ell_j$ an orientation; here the
  Delzant property for $\Delta$ is not the standard one for polygons,
  but rather a more delicate one for weighted polygons which takes
  into account the presence of the lines $\ell_j$; (iv) an
  $m_f$\--tuple of positive real numbers $(h_i)_{i=1}^{m_{f}}$ such
  that $ 0 < h_i < \textup{length}(\Delta \cap \ell_i) $ for each $i
  \in \{1,\ldots,m_f\}$.  (v) an arbitrary collection of $m_f$
  integers $(k_i)_{i=1}^{m_f}$.  Our main theorem (Theorem
  \ref{existence:thm}) says that, starting from a semitoric list of
  ingredients one can construct a $4$\--dimensional semitoric
  integrable system $(M,\, \omega,\, (J,\,H))$ such that the list of
  its invariants is equal to this semitoric list.  Moreover, $M$ is
  compact if and only the polygon in item (iii) is compact.

  With this in mind we may formulate the uniqueness theorem in
  \cite{pelayovungoc} as: two systems constructed in this fashion are
  isomorphic if and only if ingredients (i), (ii) and (iv) are
  identical for both systems and ingredients (iii) and (v) are related
  by some simple transformation. This is why, when we formulate the
  existence theorem, ingredients (iii) and (v) are given by orbits of
  respectively weighted polygons and pondered weighted polygons, under
  the action of certain groups.  Together with
  \cite[Th.~6.2]{pelayovungoc}, this gives the aforementioned
  classification (Theorem \ref{class:thm}) .


  While the construction of semitoric systems in the present paper is
  relatively self\--contained, we are indebted to the articles of
  Delzant \cite{delzant}, Atiyah \cite{atiyah} and
  Guillemin\--Sternberg \cite{gs}, in the context of Hamiltonian torus
  actions, which served as an inspiration to study the more general
  situation of integrable systems with circular Hamiltonian
  symmetry. Furthermore, many works have played an important role in
  our investigation of $4$\--dimensional semitoric systems, by serving
  as stepping stones to construct the symplectic invariants in
  \cite{vungoc} associated with semitoric systems; notably we used
  work of Dufour\--Molino \cite{dufour-molino}, Eliasson
  \cite{eliasson-these}, Duistermaat \cite{duistermaat},
  Miranda\--Zung \cite{miranda-zung} and V\~u Ng\d oc
  \cite{vungoc0},\cite{vungoc}.

  In this work, we are in a situation where the moment map $(J,H)$ is
  a ``torus fibration'' with singularities, and its base space becomes
  endowed with a singular integral affine structure. These structures
  have been studied in the context of integrable systems (in
  particular by Zung~\cite{zung}), but also became a central concept
  in the works by Symington \cite{s}, Symington\--Leung \cite{ls} in
  the context of symplectic geometry and topology, and by
  Gross\--Siebert \cite{grs1}, \cite{grs2}, \cite{grs3} and
  \cite{grs4}, among others, in the context of mirror symmetry and
  algebraic geometry.  In fact, our ingredients (i), (iii) and (iv)
  could have been expressed in terms of this affine structure. However
  ingredients (ii) and (v) do not appear in the affine
  structure. Nevertheless it is expected that these ingredients play
  an important role in the quantum theory of integrable systems. We
  hope to be able to explore these ideas in the future.

  \vspace{1.5mm}

  The paper is structured as follows: in Section \ref{uniqueness:sec}
  we recall how to construct a collection of symplectic invariants for
  a semitoric system, and state more precisely that two semitoric
  systems are isomorphic precisely when they have the same invariants;
  this was done in \cite{pelayovungoc}, and we need to review it here
  in order to state the existence theorem for semitoric systems. 
  In Section \ref{glueingsection} we explain the symplectic glueing construction
  (i.e. how to glue symplectic manifolds equipped with momentum
  maps). The
  last two sections of the paper are respectively devoted to state the
  main theorem and to prove it. One might argue that the proof is more
  informative than the statement, as it gives an {\em explicit}
  construction of all semitoric integrable systems in dimension $4$.
  \\
  \\
  {\bf Acknowledgements}. We are grateful to Denis Auroux for offering
  many insightful comments, and for pointing out the papers of Gross
  and Siebert.

  \section{Review of the uniqueness theorem for semitoric systems}
  \label{uniqueness:sec}

  We recall the definition of the invariants that we assigned to a
  semitoric integrable system in our previous paper
  \cite{pelayovungoc}, to which we refer to further details.  Then we
  state the uniqueness theorem proved therein.

  \subsection{Taylor series invariant}
  \label{taylor:sec}

  It was proven in \cite{vungoc} that a semitoric system $(M,\,
  \omega,\, F:=(J,\,H))$ has finitely many focus\--focus critical
  values $c_1,\,\dots,\,c_{m_f}$, that
  if we write $B:=F(M)$ then
   the set of regular values of
  $F$ is $\op{Int}(B)\setminus\{c_1,\dots,c_{m_f}\}$, that
  the boundary of $B$ consists of all images of elliptic
  singularities, and that the fibers of $F$ are connected. The integer
  $m_f$ was the first invariant that we associated with such a system.
  Let $i$ be an integer, with $1 \le i \le m_f$.  
  
  We assume that the
  critical fiber $ \mathcal{F}_m:=F^{-1}(c_i) $ contains only one
  critical point $m$, which according to Zung \cite{zung} is a generic condition,
  and let $\mathcal{F}$ denote the associated
  singular foliation.  Moreover,
  we will make for simplicity an even stronger
generic assumption~:
\begin{equation}
  \begin{split}
    ~&\text{If $m$ is a focus-focus critical point for } F,\\
    &\text{then } m \text{ is the unique critical point of the level set
    } J^{-1}(J(m)).
  \end{split} \nonumber
\end{equation}
A semitoric system is \emph{simple} if this genericity assumption is satisfied.

These conditions imply that the values $J(c_1),\,\dots,\,J(c_{m_f})$ are
pairwise distinct. We assume
throughout the article that the critical values $c_i$'s are
\emph{ordered} by their $J$-values~:  $J(c_1)< J(c_2) < \cdots < J(c_{m_f})$.

  By Eliasson's theorem \cite{eliasson-these}
  there exist symplectic coordinates $(x,\, y,\, \xi,\,\eta)$ in a
  neighborhood $U$ around $m$ in which $(q_1,\,q_2)$, given by
  \begin{equation}
    q_1=x\xi+y\eta, \,\, q_2=x\eta-y\xi, 
    \label{equ:cartan}
  \end{equation}
  is a momentum map for the foliation $\mathcal{F}$; here the critical
  point $m$ corresponds to coordinates $(0,\,0,\,0,\,0)$.

  Fix $A'\in \mathcal{F}_m\cap (U\setminus\{m\})$ and let $\Sigma$
  denote a small 2\--dimensional surface transversal to $\mathcal{F}$
  at the point $A'$, and let $\Omega$ be the open neighborhood of
  $\mathcal{F}_m$ which consists of the leaves which intersect the
  surface $\Sigma$.  Since the Liouville foliation in a small
  neighborhood of $\Sigma$ is regular for both $F$ and
  $q=(q_1,\,q_2)$, there is a local diffeomorphism $\varphi$ of $\R^2$
  such that $q=\varphi \circ {F}$, and we can define a global momentum
  map $\Phi=\varphi \circ{F}$ for the foliation, which agrees with $q$
  on $U$.  Write $\Phi:=(H_1,\,H_2)$ and $\Lambda_z:=\Phi^{-1}(z)$.
  Note that $ \Lambda_0=\mathcal{F}_m.  $ It follows
  from~(\ref{equ:cartan}) that near $m$ the $H_2$\--orbits must be
  periodic of primitive period $2\pi$ for any point in a (non-trivial)
  trajectory of $\mathcal{X}_{H_1}$.

  Suppose that $A \in\Lambda_z$ for some regular value $z$.  Let
  $\tau_1(z)>0$ be the time it takes the Hamiltonian flow associated
  with $H_1$ leaving from $A$ to meet the Hamiltonian flow associated
  with $H_2$ which passes through $A$, and let $\tau_2(z)\in\R/2\pi\Z$
  the time that it takes to go from this intersection point back to
  $A$, hence closing the trajectory. Write
  $z=(z_1,\,z_2)=z_1+\op{i}z_2$, and let $\op{ln} z$ for a fixed
  determination of the logarithmic function on the complex plane. Let
  \begin{equation}
    \left\{
      \begin{array}{ccl}
        \sigma_1(z) & = & \tau_1(z)+\Re(\op{ln} z) \\
        \sigma_2(z) & = & \tau_2(z)-\Im(\op{ln} z),
      \end{array}
    \right.
    \label{equ:sigma}
  \end{equation}
  where $\Re$ and $\Im$ respectively stand for the real an imaginary
  parts of a complex number.  V\~u Ng\d oc proved in
  \cite[Prop.\,3.1]{vungoc0} that $\sigma_1$ and $\sigma_2$ extend to
  smooth and single\--valued functions in a neighbourhood of $0$ and
  that the differential 1\--form $\sigma:=\sigma_1\,
  \DD{}z_1+\sigma_2\, \DD{}z_2$ is closed.  Notice that if follows
  from the smoothness of $\sigma_2$ that one may choose the lift of
  $\tau_2$ to $\R$ such that $\sigma_2(0)\in[0,\,2\pi)$. This is the
  convention used throughout.  Following \cite[Def.~3.1]{vungoc0} ,
  let $S_i$ be the unique smooth function defined around $0\in\R^2$
  such that
  \begin{eqnarray}
    \DD{}S_i=\sigma,\,\, \,\, S_i(0)=0
  \end{eqnarray}
  The Taylor expansion of $S_i$ at $(0,\,0)$ is denoted by
  $(S_i)^\infty$.
 
\begin{definition}
  The Taylor expansion $(S_i)^{\infty}$ is a formal power series in
  two variables with vanishing constant term, and we say that
  $(S_i)^\infty$ is the \emph{Taylor series invariant of $(M,\,
    \omega,\, (J,\,H))$ at the focus\--focus point $c_i$}.
\end{definition}

\subsection{Semitoric polygon invariant}
\label{semitoric:sec}

The plane $\RM^2$ is equipped with its standard affine structure with
origin at $(0,0)$, and orientation.  Let
$\textup{Aff}(2,\RM):=\textup{GL}(2,\RM)\ltimes\RM^2$ be the group of
affine transformations of $\RM^2$. Let
$\textup{Aff}(2,\ZM):=\textup{GL}(2,\ZM)\ltimes\RM^2$ be the subgroup
of \emph{integral-affine} transformations.

Let $\mathfrak{I}$ be the subgroup of $\op{Aff}(2,\,\Z)$ of those
transformations which leave a vertical line invariant, or
equivalently, an element of $\mathfrak{I}$ is a vertical translation
composed with a matrix $T^k$, where $k \in \Z$ and
\begin{eqnarray}
  T^k:=\left(
    \begin{array}{cc}
      1 & 0\\ k & 1
    \end{array}
  \right) \in \op{GL}(2,\, \Z).
\end{eqnarray}
Let $\ell\subset\R^2$ be a vertical line in the plane, not necessarily
through the origin, which splits it into two half\--spaces, and let
$n\in\Z$. Fix an origin in $\ell$.  Let $t^n_{\ell} \colon \R^2
\to \R^2$ be the identity on the left half\--space, and $T^n$ on the
right half\--space. By definition $t^n_{\ell}$ is piecewise
affine.  Let $\ell_i$ be a vertical line through the focus\--focus
value $c_i=(x_i,\,y_i)$, where $1 \le i \le m_f$, and for any tuple
$\vec n:=(n_1,\,\dots,\,n_{m_f})\in\Z^{m_f}$ we set
$t_{\vec n}:=t^{n_1}_{\ell_1}\circ\, \cdots\, \circ
t^{n_{m_f}}_{\ell_{m_f}}$.
The map $t_{\vec n}$ is piecewise affine.

\begin{definition}\label{defi:rational-convex-polygon}
  A \emph{rational convex polygon} is the convex hull of
  a discrete set of points in $\RM^2$, with the condition that each
  edge is directed along a vector with rational
  coefficients.\footnote{it is important to note that a convex polygon
    is not necessarily compact for us. A more accurate denomination
    would be a rational convex polyhedron.}
\end{definition}

Let $B_{\op{r}}:=\op{Int}(B)\setminus \{c_1,\,\ldots,\,c_{m_f}\}$, which
is precisely the set of regular values of $F$. 
Given a sign $\epsilon_i\in\{-1,+1\}$, let
$\ell_i^{\epsilon_i}\subset\ell_i$ be the vertical half line starting
at $c_i$ at extending in the direction of $\epsilon_i$~: upwards if
$\epsilon_i=1$, downwards if $\epsilon_i=-1$. Let $
\ell^{\vec\epsilon}:= \bigcup_{i=1}^{m_f}\ell_i^{\epsilon_i}.  $ In
Th.~3.8 in \cite{vungoc} it was shown that for
$\vec\epsilon\in\{-1,+1\}^{m_f}$ there exists a homeomorphism $f =
f_\epsilon\colon B \to \R^2$, modulo a left composition by a
transformation in $\mathfrak{I}$, such that $f|_{(B\setminus
  \ell^{\vec\epsilon})}$ is a diffeomorphism into its image
$\Delta:=f(B)$, which is a \emph{rational convex polygon},
$f|_{(B_r\setminus \ell^{\vec\epsilon})}$ is affine (it sends the
integral affine structure of $B_r$ to the standard structure of
$\R^2$) and $f$ preserves $J$: i.e.
\[
f(x,\,y)=(x,\,f^{(2)}(x,\,y)).
\]
$f$ satisfies further properties \cite{pelayovungoc}, which are
relevant for the uniqueness proof.  In order to arrive at $\Delta$ one
cuts $(J,\,H)(M) \subset \R^2$ along each of the vertical half-lines
$\ell_i^{\epsilon_i}$. Then the resulting image becomes simply
connected and thus there exists a global 2-torus action on the
preimage of this set. The polygon $\Delta$ is just the closure of the
image of a toric momentum map corresponding to this torus action.

We can see that this polygon is not unique. The choice of the ``cut
direction'' is encoded in the signs $\epsilon_j$, and there remains
some freedom for choosing the toric momentum map. Precisely, the choices
and the corresponding homeomorphisms $f$ are the following~:
\begin{itemize}
\item[(a)] {\em an initial set of action variables $f_0$ of the form
    $(J,\,K)$} near a regular Liouville torus in \cite[Step 2, pf. of
  Th.~3.8]{vungoc}.  If we choose $f_1$ instead of $f_0$, we get a
  polygon $\Delta'$ obtained by left composition with an element of
  $\mathfrak{I}$.  Similarly, if we choose $f_1$ instead of $f_0$, we
  obtain $f$ composed on the left with an element of $\mathfrak{I}$;
\item[(b)] {\em a tuple $\vec{\epsilon}$ of $1$ and $-1$}.  If we
  choose $\vec{\epsilon'}$ instead of $\vec{\epsilon}$ we get $
  \Delta'=t_{\vec{u}}(\Delta) $ with
  $u_i=(\epsilon_i-\epsilon'_i)/2$, by \cite[Prop. 4.1,
  expr. (11)]{vungoc}.  Similarly instead of $f$ we obtain
  $f'=t_{\vec{u}} \circ f$.
\end{itemize}

\begin{lemma}
\label{lemm:mu}
  Once $f_0$ and $\vec\epsilon$ have been fixed as in (a) and (b),
  respectively, then there exists a unique toric momentum map $\mu$ on
  $M_r:=F^{-1}(\textup{Int}{B}\setminus(\bigcup \ell_j^{\epsilon_j}))$
  which preserves the foliation $\mathcal{F}$, and coincides with
  $f_0\circ F$ where they are both defined. Then, necessarily, the
  first component of $\mu$ is $J$, and we have
  \begin{equation}
    \overline{\mu(M_r)}=\Delta.
    \label{equ:mu-delta}  
  \end{equation}
\end{lemma}
\begin{demo}
  The uniqueness follows from the fact that
  $\textup{Int}{B}\setminus(\bigcup \ell_j^{\epsilon_j})$ is simply
  connected, and~(\ref{equ:mu-delta}) follows directly from the
  construction of $\Delta$ in~\cite{vungoc}, since $\mu=f\circ F$.
\end{demo}

We sometimes call $\mu$ the (generalized) momentum 
map associated with
the polytope $\Delta$.

We need now for our purposes to formalize choices (a) and (b) in a
single geometric object.  Let $\op{Polyg}(\R^2)$ be the space of
rational convex polygons in $\R^2$. Let $\op{Vert}(\R^2)$ be the set
of vertical lines in $\R^2$.  A {\em weighted polygon of complexity
  $s$} is a triple of the form
\[
\Delta_{\scriptop{w}}=\Big(\Delta,\,
(\ell_{\lambda_j})_{j=1}^s,\, (\epsilon_j)_{j=1}^s\Big)
\]
where $s$ is a non\--negative integer, $\Delta \in \op{Polyg}(\R^2)$,
$\ell_{\lambda_j} \in \op{Vert}(\R^2)$ for every $j \in
\{1,\ldots,s\}$, and $\epsilon_j \in \{-1,\,1\}$ for every $j \in
\{1,\ldots,s\}$,
\[
\op{min}_{s \in \Delta}\pi_1(s)<\lambda_1<\ldots<\lambda_s<
\op{max}_{s \in \Delta}\pi_1(s),
\]
where $\pi_1 \colon \R^2 \to \R$ is the canonical projection
$\pi_1(x,\,y)=x$ and $\pi_1(\ell_{\lambda_j})=\lambda_j$.  For any $s \in
\N$, let $G_s:=\{-1,\,+1\}^s$ and let $\mathcal{G}:=\{T^k\,\, | \,\, k
\in \Z\}$. The group $\mathcal{G}$ acts naturally on $\RM^2$ by the
affine transformations $T^k$. Obviously, it sends a rational convex
polygon to a rational convex polygon. It corresponds to the
transformation described in (a). On the other hand, the transformation
described in (b) can be encoded by the group $G_s$ acting on the
triple $\Delta_{\scriptop{w}}$ by the formula
\begin{equation}
  (\epsilon'_j)_{j=1}^s \cdot \Big(
  \Delta,\,(\ell_{\lambda_j})_{j=1}^s,\, (\epsilon_j)_{j=1}^s\Big)=
  \Big(t_{\vec u}(\Delta),\,(\ell_{\lambda_j})_{j=1}^s,\,
  (\epsilon'_j\,\epsilon_j)_{j=1}^s\Big),\label{equ:actionGs}
\end{equation}
where $\vec u=((\epsilon_i-\epsilon'_i)/2)_{i=1}^s$.  This, however,
does not always preserve the convexity of $\Delta$, as is easily seen
when $\Delta$ is the unit square centered at the origin and
$\lambda_1=0$. However, when $\Delta$ comes from the construction
described above for a semitoric system $(J,H)$, the convexity is
preserved. Thus, we say that 

\begin{definition} \label{admissiblepolygon:def}
A weighted polygon is \emph{admissible}
when the $G_s$\--action preserves convexity.  We denote by
$\mathcal{W}\op{Polyg}_s(\R^2)$ the space of all admissible weighted
polygons of complexity $s$.
\end{definition}

The set $G_s\times \mathcal{G}$ is an abelian group, with the natural
product action.  The action of $G_s\times \mathcal{G}$ on
$\mathcal{W}\op{Polyg}_s(\R^2)$, is given by:
\[
((\epsilon'_j)_{j=1}^s,\,T^k) \cdot \Big(
\Delta,\,(\ell_{\lambda_j})_{j=1}^s,\, (\epsilon_j)_{j=1}^s\Big)=
\Big(t_{\vec u}(T^k(\Delta)),\,(\ell_{\lambda_j})_{j=1}^s,\,
(\epsilon'_j\,\epsilon_j)_{j=1}^s\Big),
\]
where $\vec u=((\epsilon_i-\epsilon'_i)/2)_{i=1}^s$.

\begin{definition} \label{defi:semitoric-polygon} We call a
  \emph{semitoric polygon} the equivalence class of an admissible
  weighted polygon under the $(G_{m_f} \times \mathcal{G})$\--action.
\end{definition}

Let $\Delta$ be a rational convex polygon obtained from the momentum
image $(J,\,H)(M)$ according to the above construction of cutting
along the vertical half-lines $\ell_1^{\epsilon_1},
\ldots,\ell_{m_f}^{\epsilon_{m_f}}$.
\begin{definition} \label{generalizedpolytope} The {\em semitoric
    polygon invariant of} $(M,\, \omega,\, (J,\,H))$ is the semitoric
  polygon equal to the $(G_{m_f} \times \mathcal{G})$\--orbit
  \begin{eqnarray} \label{polin} (G_{m_f} \times \mathcal{G})\cdot
    \Big(\Delta,\, (\ell_j)_{j=1}^{m_f},\,
    (\epsilon_j)_{j=1}^{m_f}\Big) \in
    \mathcal{W}\op{Polyg}_{m_f}(\R^2)/(G_{m_f} \times \mathcal{G}).
  \end{eqnarray}
\end{definition}

\subsection{The Volume Invariant}
\label{volumesection}

Consider a focus\--focus critical point $m_i$ whose image by $(J,\,H)$
is $c_i$, and let $\Delta$ be a rational convex polygon corresponding
to the system $(M,\, \omega,\, (J,\,H))$.  If $\mu$ is a toric
momentum map for the system $(M, \, \omega,\,(J, \, H))$ corresponding
to $\Delta$, then the image $\mu(m_i)$ is a point in the interior of
$\Delta$, along the line $\ell_i$.  We proved in \cite{pelayovungoc}
that the vertical distance
\begin{eqnarray} \label{height:eq} h_i:=\mu(m_i)-\min_{s \in \ell_i
    \cap \Delta} \pi_2(s)>0
\end{eqnarray}
is independent of the choice of momentum map $\mu$. Here $\pi_2 \colon
\R^2 \to \R$ is $\pi_2(x,\,y)=y$. The reasoning behind writing the
word ``volume'' in the name of this invariant is that it has the
following geometric interpretation: the singular manifold
$Y_i=J^{-1}(c_i)$ splits into $Y_i\cap\{H>H(m_i)\}$ and
$Y_i\cap\{H<H(m_i)\}$, and $h_i$ is the Liouville volume of
$Y_i\cap\{H<H(m_i)\}$.

\subsection{The Twisting\--Index Invariant}
\label{indexsection}

The twisting-index expresses the fact that there is, in a
neighbourhood of any focus-focus point $c_i$, a \emph{privileged toric
  momentum map} $\nu$. This momentum map, in turn, is due to the
existence of a unique hyperbolic radial vector field in a
neighbourhood of the focus-focus fiber. Therefore, one can view the
twisting-index as a dynamical invariant. Since any semitoric polygon
defines a (generalized) toric momentum map $\mu$, we will be able to
define the twisting-index as the integer $k_i\in \ZM$ such that
\[
\DD\mu = T^{k_i} \DD\nu.
\]
We could have defined equivalently the twisting-indices by comparing
the privileged momentum maps at different focus-focus points.

The precise definition of $k_i$ requires some care, which we explain
now.

Let $\Delta_{\scriptop{w}}= \Big(\Delta,\,
(\ell_j)_{j=1}^{m_f},\, (\epsilon_j)_{j=1}^{m_f}\Big)$ be as in
expression (\ref{polin}).  Let $\ell:=\ell_i^{\epsilon_i}\subset\R^2$
be the vertical \emph{half\--line} starting at $c_i$ and pointing in
the direction of $\epsilon_i\, e_2$, where $e_1,\,e_2$ are the
canonical basis vectors of $\mathbb{R}^2$.  By Eliasson's theorem,
there is a neighbourhood $W=W_i$ of the focus\--focus critical point
$m_i=F^{-1}(c_i)$, a local symplectomorphism $\phi:(\R^4,\,0)\to W$,
and a local diffeomorphism $g$ of $(\R^2,\,0)$ such that $F \circ
\phi= g\circ q$, where $q$ is given by~\eqref{equ:cartan}. Since
$q_2\circ\phi^{-1}$ has a $2\pi$\--periodic Hamiltonian flow, it is
equal to $J$ in $W$, up to a sign.  Composing if necessary $\phi$ by
$(x,\xi)\mapsto(-x,-\xi)$ one can assume that $q_2=J\circ\phi$ in $W$,
i.e. $g$ is of the form $g(q_1,\,q_2)=(q_2,\,g_2(q_1,\,q_2))$.  Upon
composing $\phi$ by
$(x,\,y,\,\xi,\,\eta)\mapsto(-\xi,\,-\eta,\,x,\,y)$, which changes
$(q_1,\,q_2)$ into $(-q_1,\,q_2)$, one can assume that $\frac{\partial
  g_2}{\partial q_1}(0)>0$.  In particular, near the origin $\ell$ is
transformed by $g^{-1}$ into the positive real axis if $\epsilon_i=1$,
or the negative real axis if $\epsilon_i=-1$.

\begin{figure}[htb]
  \begin{center}
    \includegraphics{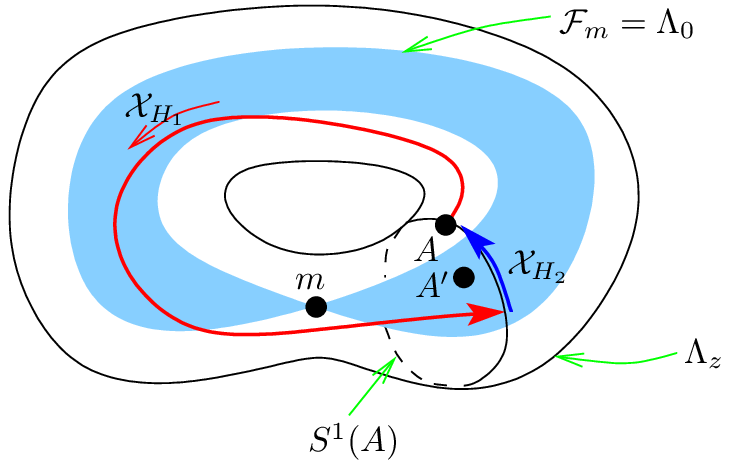}
    \caption{Singular foliation near the leaf $\mathcal{F}_m$, where
      $S^1(A)$ denotes the $S^1$\--orbit generated by $H_2$.}
  \end{center}
  \label{AF6}
\end{figure}

Let us now fix the origin of angular polar coordinates in $\R^2$ on
the \emph{positive} real axis, let $V=F(W)$ and define
$\tilde{F}=(H_1,\,H_2)=g^{-1}\circ F$ on $F^{-1}(V)$ (notice that
$H_2=J$). Recall that near any regular torus there exists a
Hamiltonian vector field $\mathcal{X}_p$, whose flow is
$2\pi$\--periodic, defined by
  $$
  2\pi\mathcal{X}_p =
  (\tau_1\circ\tilde{F})\mathcal{X}_{H_1}+(\tau_2\circ\tilde{F})\mathcal{X}_J,
  $$
  where $\tau_1$ and $\tau_2$ are functions on $\R^2\setminus\{0\}$
  satisfying~\eqref{equ:sigma}, with $\sigma_1(0)>0$. In fact $\tau_2$
  is multivalued, but we determine it completely in polar coordinates
  with angle in $[0,\,2\pi)$ by requiring continuity in the angle
  variable and $\sigma_2(0)\in[0,\,2\pi)$. In case $\epsilon_i=1$,
  this defines $\mathcal{X}_p$ as a smooth vector field on
  $F^{-1}(V\setminus\ell)$. In case $\epsilon_i=-1$ we keep the same
  $\tau_2$\--value on the negative real axis, but extend it by
  continuity in the angular interval $[\pi,\,3\pi)$. In this way
  $\mathcal{X}_p$ is again a smooth vector field on
  $F^{-1}(V\setminus\ell)$.  Let $\mu$ be the generalized toric
  momentum map associated to $\Delta$. On $F^{-1}(V\setminus \ell$),
  $\mu$ is smooth, and its components $(\mu_1,\,\mu_2)=(J,\,\mu_2)$
  are smooth Hamiltonians, whose vector fields
  $(\mathcal{X}_J,\mathcal{X}_{\mu_2})$ are tangent to the foliation,
  have a $2\pi$-periodic flow, and are \emph{a.e.}  independent. Since
  the couple $(\mathcal{X}_J,\mathcal{X}_p)$ shares the same
  properties, there must be a matrix $A\in\textup{GL}(2,\Z)$ such that
  $(\mathcal{X}_J,\mathcal{X}_{\mu_2}) = A
  (\mathcal{X}_J,\mathcal{X}_p)$. This is equivalent to saying that
  there exists an integer $k_i\in\Z$ such that
$$
\mathcal{X}_{\mu_2} = k_i\mathcal{X}_{J} + \mathcal{X}_p.
$$

It was shown in \cite[Prop.~5.4]{pelayovungoc} that $k_i$ is well
defined, i.e. does not depend on choices. The integer $k_i$ is called
the \emph{twisting index of $\Delta_{\scriptop{w}}$ at the
  focus\--focus critical value $c_i$}.  It was shown in
\cite[Lem.~5.6]{pelayovungoc} that there exists a unique smooth
function $H_p$ on $F^{-1}(V\setminus \ell)$ the Hamiltonian vector
field of which is $\mathcal{X}_p$ and such that $\lim_{m\to
  m_i}H_p=0.$ The toric momentum map $\nu:=(J,\,H_p)$ is called {\em
  the privileged momentum map for $(J,\,H)$} around the focus\--focus
value $c_i$.  If $k_i$ is the twisting index of $c_i$, one has $\DD\mu =
T^{k_i} \DD\nu$ on $F^{-1}(V)$. However, the twisting index does depend
on the polygon $\Delta$. Thus, since we want to define an invariant of
the initial semitoric system, we need to take into account the actions
of $G_s$ and $\mathcal{G}$.

If we transform the polygon $\Delta$ by a global affine transformation
in $T^r\in\mathcal{G}$ this has no effect on the privileged momentum
map $\nu$, whereas it changes $\mu$ into $T^r \mu$.  From this
characterization it follows that all the twisting indices $k_i$ are
replaced by $k_i+r$.  It was shown in \cite[Prop.~5.8]{pelayovungoc}
that if two weighted polygons $\Delta_{\scriptop{w}}$ and
$\Delta'_{\scriptop{weight}}$ lie in the same $G_{m_f}$\--orbit, then
the twisting indices $k_i,\,k'_i$ associated to
$\Delta_{\scriptop{w}}$ and $\Delta'_{\scriptop{weight}}$ at
their respective focus\--focus critical values $c_i,\,c'_i$ are equal.

For any integer $s$, consider the action of the product $G_s \times
\mathcal{G}$ on the space $\mathcal{W}\op{Polyg}_s(\R^2) \times \Z^{s}
$:
\[
((\epsilon'_j)_{j=1}^s),\,T^k) \star \Big(
\Delta,\,(\ell_{\lambda_j})_{j=1}^s,\, (\epsilon_j)_{j=1}^s,\,
(k_j)_{j=1}^{s}\Big)
=\Big(t_{\vec{u}}(T^k(\Delta)),\,(\ell_{\lambda_j})_{j=1}^s,\,
(\epsilon'_j\,\epsilon_j)_{j=1}^s, \, (k_j+k)_{j=1}^{s}\Big)
\]
where $\vec u=((\epsilon_j-\epsilon'_j)/2)_{j=1}^s$, for all integer
$j$, with $j \in \{1,\, \ldots,\, s\}$.

\begin{definition} \label{indexinv} The {\em twisting\--index
    invariant} of $(M,\, \omega,\, (J,\,H))$ is the $(G_{m_f} \times
  \mathcal{G})$\--orbit of weighted polygon pondered by twisting
  indices at the focus\--focus singularities of the system given by
  \begin{eqnarray}
    (G_{m_f} \times \mathcal{G})\star \Big(\Delta,\, (\ell_j)_{j=1}^{m_f},\,
    (\epsilon_j)_{j=1}^{m_f},\, (k_j)_{j=1}^{m_f}\Big) \in (\mathcal{W}\op{Polyg}_{m_f}(\R^2) \times \Z^{m_f})/(G_{m_f} \times \mathcal{G}).
  \end{eqnarray}
\end{definition}

\subsection{Uniqueness theorem}

To a semitoric system we assign the above list of invariants and state
the main theorem in \cite{pelayovungoc}.

\begin{definition} \label{listofinvariants} Let $(M, \, \omega, \, (J,
  \, H))$ be a $4$\--dimensional simple semitoric integrable system.  The
  {\em list of invariants of $(M, \, \omega, \, (J, \, H))$} consists
  of the following items.
  \begin{itemize}
  \item[(i)] The integer number $0 \le m_f<\infty$ of focus\--focus
    singular points.
  \item[(ii)] The $m_{f}$\--tuple $((S_i)^{\infty})_{i=1}^{m_{f}}$,
    where $(S_i)^{\infty}$ is the Taylor series of the
    $i^{\scriptop{th}}$ focus\--focus point.
  \item[(iii)] The semitoric polygon invariant, c.f. Definition
    \ref{generalizedpolytope}.
  \item[(iv)] The volume invariant, i.e. the $m_f$\--tuple
    $(h_i)_{i=1}^{m_{f}}$, where $h_i$ is the height of the
    $i^{\scriptop{th}}$ focus\--focus point.
  \item[(v)] The twisting\--index invariant, c.f. Definition
    \ref{indexinv}.
  \end{itemize}
\end{definition}

\begin{theorem}[Th.~6.2, \cite{pelayovungoc}] \label{mainthm} The two
  $4$\--dimensional simple semitoric integrable systems $(M_1, \, \omega_1,
  (J_1,\,H_1))$ and $(M_2, \, \omega_2, (J_2,\,H_2))$ are isomorphic
  if and only if the list of invariants (i)\--(v), as in Definition
  \ref{listofinvariants}, of $(M_1, \, \omega_1, (J_1,\,H_1))$ is
  equal to the list of invariants (i)\--(v) of $(M_2, \, \omega_2,
  (J_2,\,H_2))$.
\end{theorem}

\section{The symplectic glueing construction}
\label{glueingsection}

In this section we explain how to symplectically glue an arbitrary
collection of symplectic manifolds $(M_{\alpha})_{\alpha \in A}$
equipped with continuous, proper maps $F_{\alpha} \colon M_{\alpha}
\to \mathbb{R}$ to form a new symplectic manifold $M$ equipped with a
continuous, proper map which restricted to $M_{\alpha}$ is equal to
$F_{\alpha}$, c.f.  Theorem \ref{theo:glueing}.
The results of this section, while perhaps well\--known among experts,
we could not find in the literature.

\subsection{Glueing maps, glueing groupoid}

Let $A$ be an arbitrary set of indices, and let $(M_\alpha)_{\alpha\in
  A}$ be a family of sets.  Recall that the \emph{disjoint union of
  the sets $M_\alpha$}, $\alpha\in A$ is the subset of
$(\bigcup_{\alpha\in A} M_\alpha)\times A$ defined by
\[
\bigsqcup_{\alpha\in A} M_\alpha:= \{(x,\,\alpha)\, | \, x\in
M_\alpha\}.
\]
We denote by $j_\alpha$, $\alpha\in A$, the natural inclusions~:
$j_\alpha :~ M_\alpha \hookrightarrow \bigsqcup_{\alpha\in A}
M_\alpha,\,\, x \mapsto (x,\, \alpha) $.  Notice that if $B\subset A$
then $\bigsqcup_{\alpha\in B} M_\alpha\subset \bigsqcup_{\alpha\in A}
M_\alpha$.  Of course, if all $M_\alpha$'s are pairwise disjoint, as
sets, then there is a natural bijection bewteen $\bigsqcup_{\alpha\in
  A} M_\alpha$ and the usual union $\bigcup_{\alpha\in A} M_\alpha$.

If the $M_\alpha$'s are topological spaces, the disjoint union
$\bigsqcup_{\alpha\in A} M_\alpha$ is endowed with the final
topology~: the finest topology that makes the inclusions $j_\alpha$
continuous. In particular $j_\alpha(M_\alpha)$ is an open set in
$\bigsqcup_{\alpha\in A} M_\alpha$.

\begin{definition}
  A \emph{glueing map for the family} $(M_\alpha)_{\alpha\in A}$ is a
  homeomorphism $\phy:U_\alpha\fleche U_\beta$ where
  $(\alpha,\beta)\in A^2$, and $U_\alpha\subset M_\alpha$ and
  $U_\beta\subset M_\beta$ are open sets.
\end{definition}
In this text we use the standard set\--theoretical convention that the
notation $\phy$ includes the source and target sets $U_\alpha$ and
$U_\beta$; in particular the notation $\phy(x)$ implies $x\in
U_\alpha$. When required, we use the notation $U_\phy^s$ and
$U_\phy^t$ for the source and target sets of $\phy$ (assuming
$U_\phy^t=\phy(U_\phy^s)$).
\begin{definition}
  Let $\mathcal{G}$ be a collection of glueing maps for
  $(M_\alpha)_{\alpha\in A}$.  The associated \emph{glueing groupoid}
  $G$ is the groupoid generated by the set of all restrictions of all
  glueing maps $\phy\in\mathcal{G}$ to open subsets of the source
  sets, with the natural groupoid law~: $\phy_2\circ\phy_1$ exists
  whenever the image of the source set of $\phy_1$ is included in the
  source set of $\phy_2$.
\end{definition}
\begin{definition}
  We say that $\mathcal{G}$ is \emph{free}
  when there is no nontrivial $\phy\in G$ with both source and target
  in the same set $M_\alpha$.
\end{definition}

\subsection{Topological glueing}

We define now the general patching construction.  Throughout this
section, and unless otherwise stated, we do not require topological
spaces to be paracompact or Hausdorff.

\begin{definition}
  Let $(M_\alpha)_{\alpha\in A}$ be a collection of pairwise disjoint
  topological spaces, and $G$ an associated glueing groupoid. From
  this we define the set $M$, called the \emph{glueing of
    $(M_\alpha)_{\alpha\in A}$ along $G$}, as $M :=
  \bigsqcup_{\alpha\in A} M_\alpha / \sim$ where $\sim$ is the
  equivalence relation on $\bigsqcup_{\alpha\in A} M_\alpha$ defined
  by
  \[
  (x,\,\alpha) \sim (x',\beta) \ssi \left(x=x' \,\text{ or there
      exists }\, \phy\in G \,\, \textup{with}\,\, x'=\phy(x)\right).
  \]
\end{definition}
Let us check that $\sim$ is indeed an equivalence relation.  The
reflexivity is obvious. If $(x,\,\alpha)\sim(x',\,\beta)$ and
$(x,\,\alpha)\neq (x',\,\beta)$ then $\phy(x)=x'$ for some $\phy\in
G$. But $G$ is a groupoid so $\phy^{-1}\in G$ and of course
$x=\phy^{-1}(x')$, so $(x',\,\beta)\sim(x,\,\alpha)$, which proves the
symmetry property.  Finally, if $(x,\,\alpha)\sim(x',\,\beta)$ and
$(x',\,\beta)\sim(x'',\,\gamma)$ then there exist $\phy$ anf $\phy'$
in $G$ such that $\phy(x)=x'$ and $\phy'(x')=x''$. Therefore
$\phy'\circ\phy$ is well\--defined on an open neighbourhood of $x$, so
$\phy'\circ\phy\in G$, and $(x,\,\alpha)\sim(x'',\,\gamma)$, so we
have shown the transitivity property.

Here again we could have dropped the assumption that the $M_\alpha$'s
are pairwise disjoint, or we could have used a standard union instead
of a disjoint union.

The following lemma follows from the definition of the equivalence
relation.
\begin{lemma}
  \label{lemm:preimage3}
  Let $\pi:\bigsqcup_{\alpha\in A} M_\alpha \fleche M$ be the quotient
  map.  For any subset $K\subset M_\alpha$, one has
  \[
  \pi^{-1}(y_\alpha(K))=j_\alpha(K)\cup \left(\bigcup_{\phy\in G}
    j_{\alpha(\phy)}(\phy(K\cap U_\phy^s))\right),
  \]
  where it is assumed that the union is over all $\phy$ whose source
  set $U_\phy^s$ intersects $K$, and $\alpha(\phy)$ is the element in
  $A$ such that $U_\phy^t\subset M_{\alpha(\phy)}$.
\end{lemma}

\begin{lemma}
  \label{lemm:open3}
  For the natural quotient topology on $M$, the maps
  $y_\alpha=\pi\circ j_\alpha: M_\alpha\fleche M$, $\alpha\in A$ are
  open and continuous.  They are injective if and only if
  $\mathcal{G}$ is free.
\end{lemma}
\begin{proof}
  By definition of the quotient topology, the map $\pi$ is
  continuous. Hence $y_\alpha=\pi\circ j_\alpha$ is
  continuous. Finally if $U\subset M_\alpha$ is open, then if follows
  from Lemma~\ref{lemm:preimage3} that $\pi^{-1}(y_\alpha(U))$ is open
  in $\bigsqcup_{\alpha\in A} M_\alpha$. This means that $y_\alpha(U)$
  is open in $M$.

  Fix $\alpha\in A$. Let $x$ and $x'$ be elements of $M_\alpha$. If
  $y_\alpha(x)=y_\alpha(x')$ then either $x=x'$ or $\phy(x)=x'$ for
  some $\phy\in G$. The latter is ruled out by the assumption that
  there is no nontrivial $\phy\in G$ with both source and target in
  $M_\alpha$. Thus in this case $y_\alpha$ is injective. If the
  condition is violated then there exist $x\neq x'$ in $M_\alpha$ with
  $j_\alpha(x)\sim j_\alpha(x')$ so $y_\alpha$ cannot be injective.
\end{proof}

\subsection{Smooth glueing}


\begin{lemma}
  \label{lemm:smooth3}
  If all $M_\alpha$'s are smooth manifolds, all $\phy\in\mathcal{G}$
  are diffeomorphisms and $\mathcal{G}$ is free then there exists a
  unique smooth structure on $M$ for which the maps $y_\alpha$,
  $\alpha\in A$ are embeddings.
\end{lemma}

  \begin{proof}
    Let $U\subset M_\alpha$ be open and let $g:U\fleche\RM^n$ be a
    homeomorphism. By Lemma~\ref{lemm:open3}, $y_\alpha$ is a
    homeomorphism onto its image. Let $\tilde{U}=y_\alpha(U)$ and
    $\tilde{g}=g\circ ((y_\alpha)|_{U})^{-1}$. Then $\tilde{U}$ is an
    open subset of $M$ and $\tilde{g}:\tilde{U}\fleche\RM^n$ is a
    homeomorphism. This shows that any chart of $M_\alpha$ descends
    onto a chart of $M$. Obviously the union of a family of open
    covers of $M_\alpha$ for all $\alpha\in A$ descends to an open
    cover of $M$. In order to get an atlas on $M$, it remains to check
    the compatibility condition when an open set $\tilde{V}_\alpha$
    coming from an atlas of $M_\alpha$ intersects an open set
    $\tilde{V}_{\beta}$ coming from an atlas of $M_{\beta}$. Thus, let
    $(V_\alpha,g_\alpha)$, $V_\alpha\subset M_\alpha$ and
    $(V_{\beta},g_{\beta})$, $V_{\beta}\subset M_{\beta}$ be local
    charts such that $y_\alpha(V_\alpha)=y_{\beta}(V_{\beta})$ and
    $\alpha\neq\beta$. Now consider the formula, given by
    Lemma~\ref{lemm:preimage3}~:

\[
j_\alpha(V_\alpha)\cup \left(\bigcup_{\phy\in G}
  j_{\alpha(\phy)}(\phy(V_\alpha\cap U_\phy^s))\right) =
j_\beta(V_\beta)\cup \left(\bigcup_{\phy\in G}
  j_{\alpha(\phy)}(\phy(V_\beta\cap U_\phy^s))\right).
\]
Because $\mathcal{G}$ is free, any $\phy$ whose source set intersects
$V_\alpha$ \emph{and} with $\alpha(\phy)=\alpha$ must be the
identity. Hence, in the lefthand side one can ommit all $\phy$'s such
that $\alpha(\phy)=\alpha$. For the same reason, one can assume that
all $\alpha(\phy)$'s are pairwise different. Of course the analogue
observation holds for the righthand side.  Hence we can equate terms
in the unions (up to permutation). In particular there must exist some
$\phy$ with $\alpha(\phy)=\beta$ and $ j_\beta(\phy(V_\alpha\cap
U_\phy^s)) = j_\beta(V_\beta).  $ Since $j_\beta$ is injective,
$\phy(V_\alpha\cap U_\phy^s)=V_\beta$. Let $x\in V_\beta$ and
$x'=\phy^{-1}(x)\in V_\alpha$. Then $y_\alpha(x')=y_\beta(x)$,
\emph{i.e.}  $x'=y_\alpha^{-1}\circ y_\beta(x)$. Thus $
((y_\alpha)|_{V_\alpha})^{-1}\circ (y_{\beta})|_{ V_{\beta}} =
(\phy^{-1})|_{V_{\beta}}.  $ Hence the transition map for the charts
$\tilde{g}_u:=g_u\circ ((y_u)|_{V_u})^{-1}$ ($u=\alpha,\beta$) is
equal to
\begin{equation}
  \tilde{g}_\alpha\circ {\tilde{g}_{\beta}}^{-1} = g_\alpha \circ \left(((y_\alpha)|_{
      V_\alpha})^{-1}\circ (y_{\beta})|_{V_{\beta}}\right) \circ g_{\beta}^{-1} =
  g_\alpha\circ \phy^{-1}\circ g_{\beta}^{-1},
  \label{equ:transition3}
\end{equation}
which is indeed a composition of local diffeomorphisms. Thus $M$ has a
natural smooth structure.

Consider now the map $y_\alpha:M_\alpha\hookrightarrow M$. Read in a
chart $(\tilde{V}_\alpha,\tilde{g}_\alpha)$ of $M$, with
$\tilde{g}_\alpha:=g_\alpha\circ ((y_\alpha)|_{V_\alpha})^{-1}$, for
some chart $(V_\alpha,g_\alpha)$ on $M_\alpha$, it becomes $
\tilde{g}_\alpha\circ y_\alpha = (g_\alpha)|_{V_\alpha}, $ which is a
local diffeomorphism. Since we already know that $y_\alpha$ is a
homeomorphism onto its image, it is an embedding.

Conversely, if $y_\alpha$, $\alpha\in A$ have to be embeddings for
some smooth structure on $M$, then any local chart on $M_\alpha$ is
sent by $y_\alpha$ to a local chart on $M$. Thus, necessarily, we
obtain the same charts on $M$ as the ones we've just constructed.
\end{proof}

\begin{remark}
  The smooth manifold $M$ given in Lemma \ref{lemm:smooth3} \emph{is
    not necessarily a Hausdorff space}.  The definition of manifold in
  Bourbaki \cite{bourbaki2} does not require $M$ to be a Hausdorff
  topological space, or a paracomact space. These are, however,
  conditions most frequently required.  It follows from Bourbaki
  \cite{bourbaki2} that $M$ is Hausdorff if, and only if, for any two
  smooth charts $\varphi \colon U \subset M \to \R^n$, $\psi \colon V
  \subset M \to \R^n$ constructed as in the proof of Lemma
  \ref{lemm:smooth3}, we have that the graph of $\psi\circ\phy^{-1}
  \colon \phy(U\cap V) \to \psi(U\cap V)$ is closed in $\varphi(U)
  \times \psi(V) \subset \R^n \times \R^n$.
\end{remark}

\subsection{Symplectic glueing}

Unlike in the previous two sections, we shall be assuming that the
$M_{\alpha}$, $\alpha \in A$, are Haudorff, paracompact smooth
manifolds. Moreover, we will be assuming that there exist continuous,
proper maps $F_{\alpha} \colon M \to \R^n$ which can be glued together
to give rise to a proper map $F \colon M \to \R$. With the aid of $F$
we will show that the Hausdorff and paracompactness properties of the
$M_{\alpha}$ are inherited by $M$.

\begin{lemma}
  \label{lemm:symplectic3}
  If for each $\alpha\in A$, $M_\alpha$ is symplectic with symplectic
  form $\omega_\alpha$, and if all $\phy\in\mathcal{G}$ are
  symplectomorphisms (and $\mathcal{G}$ is free) then there exists a
  unique symplectic structure $\omega$ on $M$ such that $y_\alpha^*
  \omega = \omega_i, \quad \alpha\in A$.
\end{lemma}


  
  \begin{proof}
    Because (1) all $y_\alpha$'s are embeddings, (2)
    $\bigcup_{\alpha\in A} y_\alpha(M_\alpha)=M$, (3) when
    $y_\alpha(M_\alpha)$ intersects $y_\beta(M_\beta)$,
    $\alpha\neq\beta$, then $y_{\beta}^{-1}\circ (y_\alpha)=\phy$ for
    some $\phy\in G$ with $\phy^*\omega_{\beta}=\omega_\alpha$, the
    formula $y_\alpha^* \omega = \omega_\alpha$ defines a unique
    symplectic form $\omega$ on $M$.
  \end{proof}

  We can finally apply this technique in our case~:
  \begin{prop}
    \label{prop:fibration0}
    Let $(M_\alpha)_{\alpha\in A}$ be a collection of symplectic
    manifolds, each equipped with a map
    $F_\alpha:M_\alpha\fleche\RM^n$. For any $\alpha,\beta\in A$ let
    $D_{\alpha\beta}:=F_\alpha(M_\alpha)\cap F_{\beta}(M_{\beta})$ and
    assume
    \begin{enumerate}
    \item $U_\alpha:=F_\alpha^{-1}(D_{\alpha\beta})$ and
      $U_{\beta}:=F_{\beta}^{-1}(D_{\alpha\beta})$ are open.

    \item $\phy_{\alpha\beta}:U_\alpha\fleche U_{\beta}$ is a
      symplectomorphism such that $ \phy_{\alpha\beta}^* F_{\beta} =
      F_\alpha$.

    \item When $D_{\alpha\beta\gamma}:=F_\alpha(M_\alpha)\cap
      F_{\beta}(M_{\beta})\cap F_\gamma(M_\gamma)\neq \emptyset$, $
      \phy_{\beta\gamma} \circ \phy_{\alpha\beta}= \phy_{\alpha\gamma}
      \, \text{(restricted to }
      F_\alpha^{-1}(D_{\alpha\beta\gamma})\text{)}.  $
    \end{enumerate}
    Then the smooth manifold $M$ obtained by glueing the collection
    $(M_\alpha)_{\alpha\in A}$ along the set of all
    $(\phy_{\alpha\beta})$ is symplectic, and there exists a unique
    map $F:M\fleche \RM^n$ verifying $F_\alpha = F\circ y_\alpha$,
    where $y_\alpha:M_\alpha\hookrightarrow M$, $\alpha\in A$ are the
    natural symplectic embeddings.
  \end{prop}

\begin{proof}
  The third assumption (cocycle condition) implies that the
  corresponding glueing groupoid is free.
\end{proof}

\begin{theorem}[Symplectic Glueing]
  \label{theo:glueing}
  Let $(M_\alpha)_{\alpha\in A}$ be a collection of symplectic
  manifolds, each equipped with a continuous, proper map
  $F_\alpha:M_\alpha\fleche V_\alpha\subset\RM^n$, where $V_\alpha$ is
  open. For any $\alpha,\beta\in A$ let $D_{\alpha\beta}:=V_\alpha\cap
  V_\beta$ and assume
  \begin{enumerate}
  \item
    $\phy_{\alpha\beta} : F_{\alpha}^{-1}(D_{\alpha \beta})\fleche
    F_{\beta}^{-1}(D_{\alpha \beta})$ is a symplectomorphism such that
    $ \phy_{\alpha\beta}^* F_{\beta} = F_\alpha$.

  \item When $V_\alpha\cap V_{\beta}\cap V_\gamma\neq \emptyset$,
    $\phy_{\beta\gamma} \circ \phy_{\alpha\beta}=
    \phy_{\alpha\gamma}$.
  \end{enumerate}
  Then the smooth manifold $M$ obtained by glueing the collection
  $(M_\alpha)_{\alpha\in A}$ along the set of all
  $(\phy_{\alpha\beta})$ is Hausdorff, paracompact (in other words, a
  smooth manifold in the usual sense) and symplectic, and there exists
  a unique continuous, proper map $ F:M\fleche \bigcup_{\alpha\in A}
  V_\alpha\subset \RM^n $ verifying $F_\alpha = F\circ y_\alpha$,
  where $y_\alpha:M_\alpha\hookrightarrow M$, $\alpha\in A$, are the
  natural symplectic embeddings.
\end{theorem}

\begin{proof}
  The main statement is a corollary of
  Proposition~\ref{prop:fibration0} since $F^{-1}(V_\alpha\cap
  V_\beta)=F^{-1}(F(M_\alpha)\cap F(M_\beta))$ and thus the right
  handside is automatically open.
  
  Next we show that $M$ is Hausdorff.
  Let $\bar{z},\, \bar{w} \in M$, where $z,\, w \in \bigsqcup_{\alpha
    \in A} M_{\alpha}$.  There are two possibilities, that
  $F(\bar{z})=F(\bar{w})$ or that $F(\bar{z}) \neq F(\bar{w})$.  If
  $F(\bar{z})=F(\bar{w})$, then by definition of $F$
  (i.e. $F_{\alpha}=F \circ y_{\alpha}$), there exists $\alpha \in A$
  such that $z \in M_{\alpha}$ and $w \in M_{\alpha}$. Here we are
  viewing $M_{\alpha}$ as a subset of $\bigsqcup_{\alpha \in A}
  M_{\alpha}$, under the canonical identification $y_{\alpha}$.
  Because $M_{\alpha}$ is Hausdorff, there exist open sets $U_{z}
  \subset M_{\alpha}$, $U_{w} \subset M_{\alpha}$, with $z \in U_{z}$,
  $w \in U_{w}$ and $U_z \cap U_{w} =\emptyset$.  Because $M_{\alpha}$
  is open in $\bigsqcup_{\alpha \in A} M_{\alpha}$, by Lemma
  \ref{lemm:open3} we have that $\pi(U_z)$ and $\pi(U_{w})$ are open
  subsets of $M$. By construction, $\bar{z} \in \pi(U_z)$, $\bar{w}
  \in \pi(U_w)$.  It follows from the definition of $\pi$ as the
  quotient map $\bigsqcup_{\alpha \in A} M_{\alpha} \to
  M=\bigsqcup_{\alpha \in A} M_{\alpha}/\sim$, that $\pi(U_z) \cap
  \pi(U_{w})=\pi(U_z\cap U_{w})=\pi(\emptyset)=\emptyset$.

  Suppose on the other hand that $F(\bar{z})\neq F(\bar{w})$. Since
  $F(\bar{z}) \in \R^n$, $F(\bar{w}) \in \R^n$, and $\R^n$ is
  Hausdorff, there exist open sets $W_z$ and $W_{w}$ in $\R^n$ such
  that $F(\bar{z}) \in W_z$, $F(\bar{w}) \in W_{w}$ and $W_z \cap
  W_{w}=\emptyset$. Since $F$ is continuous, $F^{-1}(W_z)$ and
  $F^{-1}(W_w)$ are open. Also, by construction, $\bar{z} \in
  F^{-1}(W_z)$ and $\bar{w} \in F^{-1}(W_w)$.
  Of course $F^{-1}(W_z) \cap F^{-1}(W_w) = F^{-1}(W_z \cap
  W_{w})=\emptyset$.


  Let us show that $F$ is proper. Let $V:=\bigcup_{\alpha\in
    A}V_\alpha$. Let $K\subset V$ be compact in $V$. Since $K$ is
  compact, there exists a finite number of open balls $B_i$ of radius
  $\epsilon>0$ that cover $K$ and such that any $\overline{B_i}$ is
  included in some $V_{\alpha(i)}$, $\alpha(i)\in A$.  Let
  $\{O_\beta\}_{\beta\in B}$ be an open cover of $F^{-1}(K)$. For any
  $i$, the set $\overline{B_i}$ is compact in $V_{\alpha(i)}$; hence
  $F_\alpha^{-1}(\overline{B_i})$ is compact in $M_\alpha$. Thus
  $y_\alpha(F_\alpha^{-1}(\overline{B_i}))$ is compact in $M$, and
  hence there exists a finite subset $B_i\subset B$ such that
  $\bigcup_{\beta\in B_i} O_\beta \supset
  y_\alpha(F_\alpha^{-1}(\overline{B_i})).  $ We can conclude, using
  the fact that
  \begin{equation}
    \label{equ:equality}
    \textup{for all}\,\, U\subset V_\alpha, \quad y_\alpha(F_\alpha^{-1}(U))=F^{-1}(U),
  \end{equation}
  that $ F^{-1}(K)\subset \bigcup_{i}\bigcup_{\beta\in B_i} O_\beta, $
  which shows that $F^{-1}(K)$ is indeed compact.

  To complete the properness proof we must show that
  equality~(\ref{equ:equality}) holds. Indeed, the inclusion of sets $
  y_\alpha(F_\alpha^{-1}(U))\subset F^{-1}(U) $ follows directly from
  the equality $F\circ y_\alpha=F_\alpha$. For the converse, we come
  back to the definition of $M$. If $\bar{z}\in F^{-1}(U)$ there must
  exist some $z_\beta\in M_\beta$ such that $\pi(z_\beta)=\bar{z}$
  ($\pi$ is the quotient map of Lemma~\ref{lemm:preimage3}). Thus
  $F_\beta(z_\beta)=F(\bar{z})$. This means that $V_\alpha\cap
  V_\beta$ is not empty, and there is a symplectomorphism
  $\phy_{\beta\alpha}$ such that
  $z_\alpha:=\phy_{\beta\alpha}(z_\beta)\in M_\alpha$. This implies
  $\pi(z_\alpha)=\pi(z_\beta)=\bar{z}$. Thus
  $F(\bar{z})=F_\alpha(z_\alpha)$ which proves the inclusion $
  F^{-1}(U) \subset y_\alpha(F_\alpha^{-1}(U)).  $

  We have left to show that $M$ is a paracompact space.  We have
  previously shown that $F\colon M \to V$ is a proper map, so in
  particular, the fibers of $F$ are compact.  On the other hand, for
  each $\alpha \in A$, $M_{\alpha}$ is a manifold in the usual sense,
  and hence it is locally compact, which then implies that
  $\bigsqcup_{\alpha \in A} M_{\alpha}$ is locally compact. We claim
  that $M$ is locally compact. Indeed, let $\bar{z} \in M$, where $z
  \in M_{\alpha}$ for some $\alpha$. Because $M_{\alpha}$ is locally
  compact, there is a compact neighborhood $K_z$ of $z$ in
  $M_{\alpha}$ containing an open set $U_z$, with $z \in U_z$. Since
  $\pi$ is continuous, $\pi(K_z)$ is compact. Since $\pi$ is open,
  $\pi(U_z)$ is open, and hence $\pi(K_z)$ is a compact neighborhood
  of $\bar{z}$, and we have shown that $M$ is locally compact.

  On the other hand, a continuous, proper map between locally compact
  Hausdorff spaces is closed\footnote{Let $f:X\fleche Y$ be such a
    map. Let $A$ be closed and let $y\in \overline{F(A)}$. Since $Y$
    is Hausdorff $\{y\}$ is the intersection of closed neighborhoods
    of $y$. Since $Y$ is locally compact one can assume that one of
    these neighborhood is compact. Since $f$ is continuous and proper,
    $A\cap{f^{-1}(y)}$ is a decreasing intersection of nonempty closed
    sets in a compact, and hence is not empty. Hence $y\in f(A)$ and
    $f(A)$ is closed.}
  see \cite[Prop.~3,\,p.~16]{daverman}.  We have already shown that
  $M$ is Hausdorff and locally compact.  Hence, since $F \colon M \to
  V$ is a proper map, it is a also a closed map.

  Next we deduce the paracompactness of $M$ from the following result
  \cite[20G,\,p.~153]{willard}, \cite[Th.~1]{brandsma}: if $f \colon X
  \to Y$ is a continuous, closed surjective mapping between
  topological spaces with compact fibers, and $Y$ is paracompact, then
  $X$ is paracompact as well. We can apply this result with $X$ equal
  to $M$, $Y$ equal to $F(M) \subset \R^n$, and $f$ equal to $F \colon
  M \to F(M)$.  The map $F \colon M \to F(M)$ is continuous, closed,
  and it has compact fibers, and $F(M)$, as a subset of $\RM^n$, is
  paracompact. Hence $M$ is paracompact.  This concludes the proof of
  the proposition.
\end{proof}

\section{Main Theorem: statement}

Again we equip the plane $\RM^2$ with its standard affine structure
with origin at $(0,0)$, and orientation.

\subsection{Delzant semitoric polygons}

Let $\Delta\in\op{Polyg}(\R^2)$ be a convex rational polygon in
$\RM^2$, as in Definition~\ref{defi:rational-convex-polygon}. Recall
that in our terminology, $\Delta$ is not necessarily compact. We call
a vertex a point in the boundary $\partial\Delta$ where the meeting
edges are not colinear. We shall make the following assumption
\begin{itemize}
\item[(a1)] The intersection of $\Delta$ with a vertical line is
  either compact or empty.
\end{itemize}
Consider such a vertical line intersecting the polytope. If the
intersection is not just a point, then it is a vertical segment. The
top end of this segment is said to belong to the \emph{top-boundary}
of $\Delta$.

To each vertex $z$ of $\Delta$ we associate a couple $\mathcal{B}_z$
of primitive integral vectors starting at $z$ and extending along the
direction of the edges meeting at $z$, in the order that makes them
oriented.  Then $\mathcal{B}_z$ defines a $\ZM$\--basis of
$\ZM^2\subset\RM^2$ when, viewed as a $2\times 2$ matrix, its
determinant is equal to $1$.

Let $s\in\NM^*$ and let $(\lambda_1,\dots,\lambda_s)\in\RM^s$ with
$\lambda_1<\cdots<\lambda_s$.  As before $\ell_{\lambda_j}$ is the
vertical line $\{x=\lambda_j\}$. We are interested only in the
following case
\begin{itemize}
\item[(a2)] The vertical lines $\ell_{\lambda_j}$, $j=1,\dots,s$
  intersect the top-boundary of $\Delta$.
\end{itemize}
Let $T$ be the linear transformation acting as the matrix
\[
T :=T^1=
\begin{pmatrix}
  1 & 0\\
  1 & 1
\end{pmatrix}.
\]
\begin{definition} \label{delzantweightedpolygon:def} Let $z$ be a
  vertex of the polygon $\Delta$ and $(u,\,v)=\mathcal{B}_z$. The
  point $z$ is called
  \begin{itemize}
  \item a \emph{Delzant corner} when there is no vertical line
    $\ell_{\lambda_j}$ through it and $\det(u,\,v)=1$,
  \item a \emph{hidden Delzant corner} when there is a vertical line
    $\ell_{\lambda_j}$ through it, it belongs to the top-boundary, and
    $\det(u,\,Tv)=1$.
  \item a \emph{fake corner} when there is a vertical line
    $\ell_{\lambda_j}$ through it, it belongs to the top-boundary, and
    $\det(u,\,Tv)=0$.
  \end{itemize}
\end{definition}

For the following lemma recall the definition of admissible weighted
polygon, c.f. Definition \ref{admissiblepolygon:def}.

\begin{lemma}
  \label{lemm:admissible}
  Let $\Delta$ be a convex rational polygon equipped with a set of
  vertical lines $(\ell_{\lambda_1},\dots,\ell_{\lambda_s})$, such
  that the assumptions (a1) and (a2) are satisfied. Suppose moreover
  that
  \begin{itemize}
  \item any point in the top-boundary that belongs to some vertical
    line $\ell_{\lambda_j}$ is either a hidden Delzant corner or a
    fake corner;
  \item any other vertex of $\Delta$ is a Delzant corner.
  \end{itemize}
  Then the triple
  \[
  \left(\Delta, (\ell_{\lambda_j})_{j=1}^s, (1,\dots,1)\right)
  \]
  is an admissible weighted polygon.
\end{lemma}
\begin{proof}
  We need to show that the convexity is preserved under the
  $G_s$\--action. This amounts to show that for any $j=1,\dots,s$, the
  polygon $t_{\vec{e_j}}(\Delta)$ is convex, where
  $(\vec{e_1},\,\dots,\,\vec{e_s})$ is the canonical basis of
  $\ZM^s$. Since $t_{\vec{e_j}}$ is affine on both half-spaces
  delimited by the vertical line $\ell_{\lambda_j}$, it suffices to
  show that $t_{\vec{e_j}}(\Delta)$ is locally convex near the points
  where $\ell_{\lambda_j}$ meets the boundary $\partial\Delta$.

  We let $\{a,z\}=\ell_{\lambda_j}\cap\partial\Delta$ and assume $z$
  lies on the top boundary. By assumption, $z$ is either a hidden
  Delzant corner or a fake corner. Let us consider the vectors
  $(u,\,v)=\mathcal{B}_z$. Because $z$ belongs to the top\--boundary, the
  vector $u$ must be directed to the lefthand side of $z$ and $v$ to
  the righthand side.  Since the transformation $t_{\vec{e_j}}$ acts
  only on the right half-space (and there it acts as $T$), the
  transformed edges of $t_{\vec{e_j}}(\Delta)$ at $z$ are directed
  along $(u,\,Tv)$. By assumption $\det(u,\,Tv)$ is either $0$ or $1$,
  which implies local convexity at $z$.

  Now consider the ``bottom boundary'' at the point $a$. By assumption
  the polygon is already locally convex at $a$ (which means
  $\det(u,v)\geq 0$), and a quick calculation shows that the action of
  $t_{\vec{e_j}}$ may only make it even ``more'' convex.
\end{proof}
It is easy to see that the properties of the lemma are preserved by
the $\mathcal{G}$-action. Thus we can state the following definition.

\begin{definition}
  \label{defi:delzant-semitoric}
  Let $[\Delta_{\scriptop{w}}]$ be a semitoric polygon as in
  Definition \ref{defi:semitoric-polygon}, and suppose that
 $\Delta_{\scriptop{w}}$ is a representative of the form
  $\Big(\Delta,\,(\ell_{\lambda_j})_{j=1}^s,\,
  (\epsilon_j)_{j=1}^s\Big)$ with all
  $\epsilon_{j}$'s equal to $+1$. Then $[\Delta_{\scriptop{w}}]$ is called a
  \emph{Delzant semitoric polygon} (\emph{of complexity $s$}) if the polygon
  $\Delta$ equipped with the vertical lines $\ell_{\lambda_j}$
  satisfies the hypothesis of Lemma~\ref{lemm:admissible}.
\end{definition}

We denote by $ \mathcal{D}\textup{Polyg}_s(\R^2) \subset
\mathcal{W}\textup{Polyg}_s(\R^2)/G_s \times \mathcal{G} $ the space
of Delzant semitoric polygons of complexity $s$, where $s < \infty$.

The following observation is a consequence of the construction of the
homeomorphism $f$ in Section \ref{semitoric:sec}.

\begin{lemma} \label{2:prop} The semitoric polygon in item (iii) of
  Definition \ref{listofinvariants} is a Delzant semitoric polygon.
\end{lemma}

In addition, note also that for any representative $\Delta$ of the
semitoric polygon $[\Delta_{\scriptop{w}}]$ in Definition \ref{listofinvariants},
and for each $i \in \{1,\ldots,\,m_f\}$ as in item (iv) of Definition
\ref{listofinvariants}, the height $h_i$ satisfies the inequality
\begin{eqnarray} \label{height:eq2} 0 < h_i < \textup{length}(\Delta
  \cap \ell_i).
\end{eqnarray}
This is because by (\ref{height:eq}) we have $h_i:=\mu(m_i)-\min_{s
  \in \ell_i \cap \Delta} \pi_2(s)$, where $\mu$ is a toric momentum
map for the system $(M, \, \omega,\,(J, \, H))$ corresponding to
$\Delta$.  Now, since $\mu(m_i)$ is a point in the interior of
$\Delta$, along the line $\ell_i$, expression (\ref{height:eq2})
follows.

\subsection{Main Theorem}

The following definition describes a collection of abstract
ingredients.  As we will see in the theorem following the definition,
each such a list of elements determines one, and one only one,
integrable system on a symplectic $4$\--manifold (which is not
necessarily a compact manifold, but we can characterize precisely when
it is in terms of one of the ingredients of the list). Moreover, this
integrable system is of semitoric type.

In the definition the term $\R[[X, \,Y]]$ refers to the algebra of
real formal power series in two variables, and $\R[[X,\,Y]]_ 0$ is the
subspace of such series with vanishing constant term, and first term
$\sigma_1\,X+\sigma_2\,Y$ with $\sigma_2 \in [0,\,2\, \pi)$.

\begin{definition} \label{listofingredients} A {\em semitoric list of
    ingredients} consists of the following items.
  \begin{itemize}
  \item[(i)] An integer number $0 \le m_f<\infty$.
  \item[(ii)] An $m_{f}$\--tuple of Taylor series
    $((S_i)^{\infty})_{i=1}^{m_{f}} \in (\R[[X,\,Y]]_ 0)^{m_f}$.
  \item[(iii)]

    A Delzant semitoric polygon $[\Delta_{\scriptop{w}}]$ of complexity
    $m_f$, as in Definition~\ref{defi:delzant-semitoric}.
    
        We denote the
    representative $\Delta_{\scriptop{w}}$ of $[\Delta_{\scriptop{w}}]$ by $\Big(\Delta,\,
    (\ell_{\lambda_j})_{j=1}^{m_f},\, (\epsilon_j)_{j=1}^{m_f}\Big)$.

  \item[(iv)] An $m_f$\--tuple of numbers $(h_j)_{j=1}^{m_{f}}$ such
    that $ 0 < h_j < \textup{length}(\Delta \cap \ell_i) $ for each $j
    \in \{1,\ldots,m_f\}$.
  \item[(v)] A $(G_{m_f} \times \mathcal{G})$\--orbit of
    $(\Delta_{\scriptop{w}},(k_j)_{j=1}^{m_f})$, where $(k_j)_{j=1}^{m_f}$ is a
    collection of integers.
  \end{itemize}
\end{definition}

Now we are ready to state the main theorem, the proof of which is
contructive and, in view of Section \ref{uniqueness:sec} and Lemma
\ref{2:prop}, gives a recipe to construct all semitoric integrable
systems up to isomorphisms.

\begin{theorem} \label{existence:thm} For each semitoric list of
  ingredients, as in Definition \ref{listofingredients}, there exists
  a $4$\--dimensional simple semitoric integrable system $(M,\, \omega,\,
  (J,\,H))$, such that the list of invariants (i)\--(v) of $(M,\,
  \omega,\, (J,\,H))$ as in Definition \ref{listofinvariants} is equal
  to this list of ingredients.  Moreover, $M$ is compact if and only
  the polygon in (iii) is compact.
\end{theorem}

\subsection{Classification of $4$\--dimensional semitoric systems}

Consequently, putting Theorem \ref{existence:thm} together with
Theorem \ref{mainthm} proved in \cite{pelayovungoc}, we obtain the
classification of integrable systems in symplectic $4$\--manifolds.

\begin{theorem}[Classification of $4$\--dimensional semitoric
  integrable systems]
  \label{class:thm}
  For each semitoric list of ingredients, as in Definition
  \ref{listofingredients}, there exists a $4$\--dimensional simple semitoric
  integrable system with list of invariants equal to this list of
  ingredients, c.f.  Definition \ref{listofinvariants}.  Moreover, two
  $4$\--dimensional simple semitoric integrable systems are isomorphic if,
  and only if, they are constructed from the same list of ingredients.
\end{theorem}

\section{Proof of Main Theorem}

Let $\Big(\Delta,\,(\ell_{\lambda_j})_{j=1}^s,\,
(\epsilon_j)_{j=1}^s\Big)$ be a representative of $[\Delta_{\scriptop{w}}]$ with
all $\epsilon_{j}$'s equal to $+1$. The strategy is to use the glueing
procedure of Section~\ref{glueingsection} in order to obtain a
semitoric system by constructing a suitable singular torus fibration
above $\Delta\subset\RM^2$.

For $j=1,\dots, m_f$, let $c_j\in\RM^2$ be the point with coordinates
\begin{equation}
  c_j=(\lambda_j,\,h_j+\min(\pi_2(\Delta\cap\ell_{\lambda_j}))).
  \label{equ:cuts}
\end{equation}
Because of the assumption on $h_j$, all points $c_j$ lie in the
interior of the polygon $\Delta$. We call these points \emph{nodes}.
We denote by $\ell_j^+$ the vertical half-line through $c_j$ pointing
upwards. We call these half-lines \emph{cuts}.

\ouf

We have divided the proof of the theorem in a preliminary step, three
intermediate steps and a conclusive step. In the preliminary step we
construct a convenient covering of the polygon $\Delta$.

Then we proceed as follows.  First we construct a ``semitoric system''
over the part of the polygon away from the sets in the covering that
contain the cuts $\ell_j^+$; then we attach to this ``semitoric
system'' the focus\--focus fibrations i.e.  the models for the systems
in a small neighborhood of the nodes.  Third, we continue to glue the
local models in a small neighborhood of the cuts. The ``semitoric
system'' is given by a proper toric map only in the preimage of the
polygon away from the cuts. We use the results of Section
\ref{glueingsection} as a stepping stone throughout.

Finally we recover the smoothness of the system and observe that the
invariants of the system are precisely the ingredients we started
with.

\paragraph{Preliminary stage. \emph{A convenient covering.}---}
We construct an open cover of the polygon. Because of the discreteness
of the set of vertices of the polygon, and the local compactness of
$\RM^2$, we one can find an open cover $(\Omega_\alpha)_{\alpha\in A}$
of $\Delta$ such that the following three properties hold: there
exists $\rho>0$ such that all $\Omega_\alpha$'s are integral-affine
images of the open cube $C:=I^2$ with $I=:]-\rho,\rho[$, i.e for every
$\alpha\in A$ there exists $R_\alpha\in\textup{Aff}(2,\ZM)$, such that
$\Omega_\alpha = R_\alpha(C)$; each vertex of the polygon, and each
node, is contained in only one open set $\Omega_\alpha$; two open sets
containing a vertex or a node never intersect each other.  In fact, if
$$
C_{\scriptop{e}} := C\cap \{y\geq 0\},\,\,\,\,\,C_{\scriptop{ee}}:=C
\cap\{x\geq 0\}\cap \{y\geq0\},
$$
one can assume that, for any $\alpha\in A$, \emph{(1)} if
$\Omega_\alpha$ intersects $\partial\Delta$ but does not contain any
vertex then $\Omega_\alpha\cap\Delta = R_\alpha(C_{\scriptop{e}})$,
and that \emph{(2)} if $\Omega{_\alpha}$ contains a Delzant corner,
then $\Omega_\alpha\cap\Delta = R_\alpha(C_{\scriptop{ee}})$.  The
first case holds since along any edge one can find a primitive vector,
and complete it to a $\ZM$\--basis of $\ZM^2$. It remains to compose
by a suitable translation to position the image of $C_{\scriptop{e}}$
at the right place. The second case is similar, since at a Delzant
corner the primitive vectors of the meeting edges form a $\ZM$\--basis
of $\ZM^2$, c.f. Definition \ref{delzantweightedpolygon:def}.

\paragraph{First stage. \emph{Away from the cuts.}---}
Let $A'\subset A$ be the subset obtained by removing all indices
intersecting the cuts. We construct a semitoric system above
$\bigcup_{\alpha\in A'}\Omega_\alpha$, by glueing the following local
models. Let $\D$ be the open disk in $\op{T}^*\RM=\RM^2$ of radius
$\sqrt{2\rho}$, centered at the origin. Consider the following models:
the regular model~: $M_{\scriptop{r}}:=\T^2\times C\subset T^*\T^2$
with momentum
map $$F_{\scriptop{r}}(x_1,\,x_2,\,\xi_1,\,\xi_2):=(\xi_1,\,\xi_2);$$
the tranversally elliptic model~: $M_{\scriptop{e}}:=(\T^1\times
I)\times \D\subset \op{T}^*\T^1\times \op{T}^*\RM$, with momentum map
  $$F_{\scriptop{e}}(x_1,\,\xi_1,\,x_2,\,\xi_2):=(\xi_1,\, (x_2^2+\xi_2^2)/2);$$
  and the elliptic\--elliptic model~: $M_{\scriptop{ee}}:=\D\times
  \D\subset \op{T}^*\RM\times \op{T}^*\RM$, with momentum map
  $$F_{\scriptop{ee}}(x_1,\,\xi_1,\,x_2,\,\xi_2):=((x_1^2+\xi_1^2)/2,\,(x_2^2+\xi_2^2)/2).$$
  Observe that $F_{\scriptop{r}}(M_{\scriptop{r}})=C$,
  $F_{\scriptop{e}}(M_{\scriptop{e}})= C_{\scriptop{e}}$, and
  $F_{\scriptop{ee}}(M_{\scriptop{ee}})=C_{\scriptop{ee}}$. Notice
  also that these models are all toric, in the sense that the momentum
  maps generate an effective hamiltonian $\T^2$ action. What's more,
  these momentum maps are proper for the topology induced on their
  images.

  Given any $\Omega_\alpha$, $\alpha\in A'$, we obtain a (singular)
  Lagrangian momentum map over $\Omega_\alpha$, whose image is
  precisely $\Omega_\alpha\cap\Delta$ by the following simple rule~:
  \emph{(a)} If $\Omega_\alpha$ contains no boundary points of
  $\Delta$ and no nodes, then we choose $M_\alpha:=M_{\scriptop{r}}$,
  with momentum map $F_\alpha:=R_\alpha\circ F_{\scriptop{r}}$;
  \emph{(b)} If $\Omega_\alpha$ interects $\partial\Delta$ but does
  not contain vertices, we choose $M_\alpha:=M_{\scriptop{e}}$, with
  momentum map $F_\alpha:=R_\alpha\circ F_{\scriptop{e}}$.  \emph{(c)}
  If $\Omega_\alpha$ contains a Delzant coner, we choose
  $M_\alpha:=M_{\scriptop{ee}}$, with momentum map
  $F_\alpha:=R_\alpha\circ F_{\scriptop{ee}}$.

  We describe now the transition functions~: when
  $\Delta_{\alpha\beta}:=\Omega_\alpha\cap \Omega_{\beta}\neq
  \emptyset$, we want to define a symplectomorphism
  \begin{eqnarray} \label{symplecto:eq} \phy_{\alpha\beta}:
    F_\alpha^{-1}(\Delta_{\alpha\beta})\fleche
    F_\beta^{-1}(\Delta_{\alpha\beta}) \,\,\,\,\,\,\, \textup{such
      that}\,\,\,\, \phy_{\alpha\beta}^* F_\beta=F_\alpha.
  \end{eqnarray}
  For this we use the following notation~: when
  $R\in\textup{Aff}(2,\,\ZM)$, we denote by $\tilde{R}$ the
  symplectomorphism $ \tilde{R}:\T^2\times\RM^2 (=\op{T}^*\T^2) \to
  \T^2\times\RM^2 $ given by $(x,\, \xi) \mapsto ((\trsp
  \op{d}\!R)^{-1}x,\, R\xi)$, where $\op{d}\!R$ is the linear part of
  $R$.  Remark that $\xi\circ \tilde{R}=R\circ \xi$.
  \\
  \\
  {\em Case 1.} If both $F_\alpha$ and $F_\beta$ are regular models,
  we let
  \begin{equation}
    \phy_{\alpha\beta}:=\tilde{R}_\beta^{-1}\, \tilde{R}_\alpha.
    \label{equ:rr}
  \end{equation}
  Then $F_\beta\circ\phy_{\alpha\beta}=R_\beta \circ F_{\scriptop{r}}
  \circ \phy_{\alpha\beta} = F_{\scriptop{r}} \circ
  \tilde{R}_\beta\circ \phy_{\alpha\beta} = F_{\scriptop{r}} \circ
  \tilde{R}_\alpha = F_\alpha$, i.e. (\ref{symplecto:eq}) holds.
  \\
  \\
  {\em Case 2.} If $F_\alpha$ is regular and $F_\beta$ is
  transversally elliptic, we introduce the symplectomorphism
  (symplectic polar coordinates)
  \[
  \begin{split}
    \phy_{\scriptop{re}} : M_{\scriptop{r}} \cap
    (\T^1\times\RM)\times(\T^1\times\RM_+^*) &
    \fleche (\T^1\times\RM)\times (\RM^2\setminus\{0\}) \cap M_{\scriptop{e}}  \\
    (x_1,\,\xi_1,\,x_2,\, \xi_2) & \mapsto (x_1,\,\xi_1,\,
    \sqrt{2\,\xi_2}\cos(x_2),\,-\sqrt{2\,\xi_2}\sin(x_2)).
  \end{split}
  \]
  Notice that $\phy_{\scriptop{re}}^*F_{\scriptop{e}} =
  F_{\scriptop{r}}$. Thus we can define
  \begin{equation}
    \phy_{\alpha\beta}:= \phy_{\scriptop{re}}\circ
    \tilde{R}_\beta^{-1}\tilde{R}_\alpha.\label{equ:re}
  \end{equation}
  We have $F_\beta\circ \phy_{\alpha\beta}=R_\beta \circ
  F_{\scriptop{e}}\circ \phy_{\scriptop{re}} \circ
  \tilde{R}_\beta^{-1}\tilde{R}_\alpha=R_\beta\circ F_{\scriptop{r}}
  \circ \tilde{R}_\beta^{-1}\tilde{R}_\alpha = F_{\scriptop{r}}\circ
  \tilde{R}_\alpha = F_\alpha$, i.e. (\ref{symplecto:eq}) holds.
  \\
  \\
  {\em Case 3}. Similarly, if $F_\alpha$ is regular and $F_\beta$ is
  elliptic\--elliptic, we introduce the symplectomorphism
  \[
  \begin{split}
    \phy_{\scriptop{ree}} : M_{\scriptop{r}} \cap
    (\T^1\times\RM_+^*)\times(\T^1\times\RM_+^*) &
    \fleche (\RM^2\setminus\{0\})\times (\RM^2\setminus\{0\}) \cap M_{\scriptop{ee}}  \\
    (x_1,\,\xi_1,\,x_2,\,\xi_2) & \mapsto \left(\begin{split}
        \sqrt{2\,\xi_1}\cos(x_1),\, &-\sqrt{2\,\xi_1}\sin(x_1),\,\\
        \sqrt{2\,\xi_2}\cos(x_2),\, &-\sqrt{2\,\xi_2}\sin(x_2).
      \end{split}\right)
  \end{split}
  \]
  Again $\phy_{\scriptop{ree}}^*F_{\scriptop{ee}}=F_{\scriptop{r}}$,
  and if we define
  \begin{equation}
    \phy_{\alpha\beta}:= \phy_{\scriptop{ree}}\circ
    \tilde{R}_\beta^{-1}\, \tilde{R}_\alpha,
    \label{equ:ree}
  \end{equation}
  (\ref{symplecto:eq}) holds.
  \\
  \\
  {\em Case 4}. If both $F_\alpha$ and $F_\beta$ are transversally
  elliptic models, then the affine map
  $R_{\alpha\beta}:=R_\beta^{-1}R_\alpha$ is an oriented
  transformation that preserves the upper half\--plane. Thus the
  horizontal axis is globally preserved, and the vector $e_1=(1,\,0)$
  is an eigenvector of $\op{d}\!R_{\alpha\beta}$. Since
  $\op{d}\!R_{\alpha\beta}\in\textup{SL}(2,\, \ZM)$, it is of the form
  $$T_k:=\begin{pmatrix} 1 & k\\0 & 1
  \end{pmatrix}$$ for some $k\in\ZM$.  Hence
  $R_{\alpha\beta}=\tau_u\circ T_k$ where $\tau_u$ is the translation
  by a horizontal vector $u=(u_1,\, 0)$.  Consider the
  symplectomorphism $\bar{R}_{\alpha\beta}(x_1,\, \xi_1,\, x_2,\,
  \xi_2):=(x_1',\, \xi_1',\, \,x_2',\,\xi_2')$ of $\op{T}^*\T^1\times
  \op{T}^*\RM$ given by
  \begin{equation*}
    \label{equ:S}
    \left\{
      \begin{aligned}
        x_1'&=x_1\\
        \xi_1'&= \xi_1+k (x_2^2+\xi_2^2)/2 + u_1\\
        (x_2'+\op{i}\,\xi_2')&=\textup{e}^{\op{i}\,kx_1}(x_2+i\xi_2).
      \end{aligned}\right.
  \end{equation*}
  Observe that $F_{\scriptop{e}}\circ
  \bar{R}_{\alpha\beta}=R_{\alpha\beta}\circ F_{\scriptop{e}}$. Now we
  define
  \begin{equation}
    \phy_{\alpha\beta} := {\mbox{$\bar{R}$}_{\alpha\beta}}|_{F_\alpha^{-1}(\Delta_{\alpha\beta})},
    \label{equ:ee}
  \end{equation}
  and we verify $F_\beta\circ \bar{R}_{\alpha\beta}=R_\beta
  F_{\scriptop{e}}\circ\bar{R}_{\alpha\beta}= R_\beta R_{\alpha\beta}
  F_{\scriptop{e}}= R_\alpha F_{\scriptop{e}}=F_\alpha$, hence
  (\ref{symplecto:eq}) holds.
  \\
  \\
  {\em Case 5}.  If $F_\alpha$ is a transversally elliptic model,
  while $F_\beta$ is elliptic\--elliptic, then, as in the previous
  case, the intersection $\Delta_{\alpha\beta}$ contains a portion of
  an edge, but not the vertex itself. This edge is mapped by $R_\beta$
  from either the horizontal or vertical positive axis. Suppose for
  simplicity that it is the horizontal axis.  As before, the affine
  map $R_{\alpha\beta}$ defined in Case 4 is an oriented
  transformation that either preserves the upper half\--plane, and
  thus one can construct a symplectomorphism $\bar{R}_{\alpha\beta}$
  of $\op{T}^*\T^1\times \op{T}^*\RM$ such that $F_{\scriptop{e}}\circ
  \bar{R}_{\alpha\beta}=R_{\alpha\beta}\circ
  F_{\scriptop{e}}$. Introduce the symplectomorphism
  \[
  \begin{split}
    \phy_{\scriptop{eee}} : M_{\scriptop{e}} \cap
    (\T^1\times\RM_+^*)\times\RM^2 &
    \fleche (\RM^2\setminus\{0\})\times \RM^2 \cap M_{\scriptop{ee}}  \\
    (x_1,\,\xi_1,\,\,x_2,\,\xi_2) & \mapsto
    (\sqrt{2\,\xi_1}\cos(x_1),\, -\sqrt{2\,\xi_1}\sin(x_1),\,
    x_2,\,\xi_2).
  \end{split}
  \]
  Notice that
  $F_{\scriptop{ee}}\circ\phy_{\scriptop{eee}}=F_{\scriptop{e}}$ and,
  whenever both are defined,
  $\phy_{\scriptop{eee}}=\phy_{\scriptop{ree}}\circ\phy_{\scriptop{re}}^{-1}$. We
  define
  \begin{equation}
    \phy_{\alpha\beta}:=\phy_{\scriptop{eee}}\circ \bar{R}_{\alpha\beta},
    \label{equ:eee}  
  \end{equation}
  and verify now routinely that
  $F_\beta\circ\phy_{\alpha\beta}=F_\alpha$, i.e. (\ref{symplecto:eq})
  also holds in this case.

  \normalfont

  We have defined the transition maps $\phy_{\alpha\beta}$ in the five
  cases~(\ref{equ:rr}), (\ref{equ:re}), (\ref{equ:ree}),
  (\ref{equ:ee}), and~(\ref{equ:eee}), and verified that equation
  (\ref{symplecto:eq}) holds for each of them. In fact one should also
  mention that for the non-symmetric cases~(\ref{equ:re}),
  (\ref{equ:ree}), and~(\ref{equ:eee}), we let
  $\phy_{\beta\alpha}:=\phy_{\alpha\beta}^{-1}$ (this is automatic for
  the symmetric cases~(\ref{equ:rr}) and~(\ref{equ:ee})).  Then it is
  easy to verify that the cocycle condition if fulfilled. Namely, when
  the triple intersection
  $\Omega_{\alpha\beta}\cap\Omega_{\beta\gamma}\cap\Omega_{\gamma\alpha}$
  is not empty, then
  \[
  \phy_{\gamma\alpha}\circ\phy_{\beta\gamma}\circ\phy_{\alpha\beta} =
  \textup{Id}.
  \]

  Thus we can apply the glueing construction,
  c.f. Theorem~\ref{theo:glueing}, and obtain a symplectic manifold
  $M_{A'}$ with a surjective map
$$
F_{A'}:M_{A'}\fleche \bigcup_{\alpha\in A'}\Omega_\alpha\subset\RM^2
$$
and, for each $\alpha\in A' \subset A$, there is a symplectic
embedding $\iota_\alpha: M_\alpha\hookrightarrow M_{A'}$ such that
$\iota_\alpha^* F_{A'}=F_\alpha$.  Since all $F_\alpha$ are proper
smooth toric momentum maps, so is $F_{A'}$.

\paragraph{Second stage. \emph{Attaching focus\--focus fibrations.}---}
Fix an integer $i$, with $1 \le i \le m_f$.  Using the classification
result of~\cite{vungoc0}, one can construct a focus-focus model
associated with an arbitrary Taylor series invariant. Precisely, for
each node $c_i$, there exists a symplectic manifold $M_{i}$ equipped
with a smooth map $F_{i}:M_{i}\fleche C$ such that the symplectic
invariant of the induced singular foliation is precisely the Taylor
series $S^\infty$.  Using the result of~\cite{vungoc}, one can
construct a continuous map $\mu_i:M_i\fleche D_i$, where
$D_i\subset\RM^2$ is some simply connected open set around the origin,
that is a smooth proper toric momentum map outside $\mu_i^{-1}(\ell)$,
where $\ell:=\{(0,\,y)\, | \,y\geq 0\}$. In fact $\mu_i=g_i\circ F_i$,
for some homeomorphism $g_i:C\fleche D_i$ that is smooth outside
$\ell$, and which preserves the first component~: it is of the form
$$
g_i(x,\,y)=(x,\,f_i(x,\,y)).
$$
This construction depends on the choice of a local toric momentum map
for the fibration over $C\setminus\ell$. Here we choose the privileged
momentum map as defined in Section \ref{indexsection}.  We are now in
position to add to the index set $A'$ all the indices $\alpha\in A$
corresponding to the nodes, and thus defining a new index set $A''$.
If $\Omega_\alpha$ contains the node $c_i$, we let $R_\alpha$ be the
matrix $T_{k_i}$ left-composed by the translation from the origin to
the node $c_i$. Here $k_j$ is the integer given as ingredient (v) in
the list. We may assume that $\Omega_\alpha=R_\alpha(D_i)$.  Then we
choose $M_\alpha:=M_i$ with momentum map $F_\alpha:=R_\alpha\circ
\mu_i$.

\begin{figure}[htb]
  \begin{center}
    \includegraphics{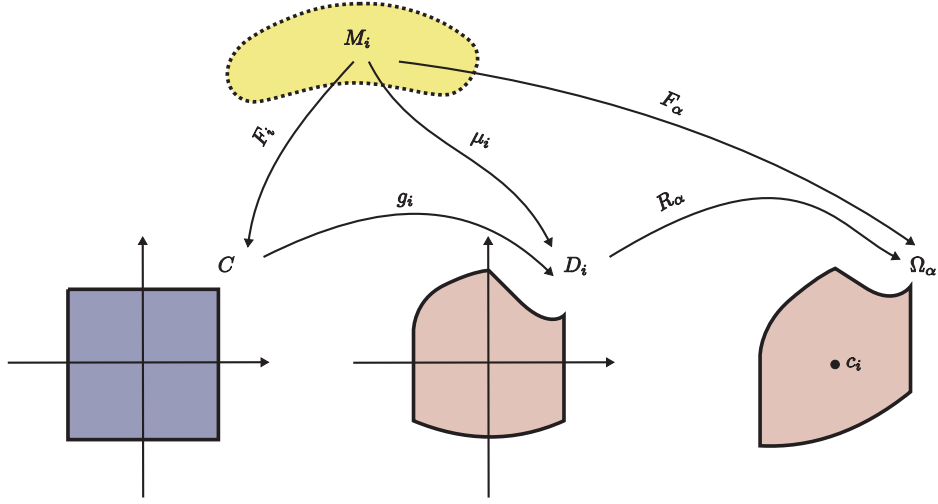}
    \caption{The pieces $M_i$ and the chart diagrams for
      $F_{\alpha},\,F_i,\, g_i$ and $R_{\alpha}$.}
  \end{center}
\end{figure}

By making $\rho$ small enough, one may assume that all $\Omega_\beta$,
$\beta\in A'$, intersecting an open set $\Omega_\alpha$ containing a
node carry regular models. Thus we need to define transition functions
between a regular model and a focus-focus model. On
$\Delta_{\alpha\beta}:=\Omega_\alpha\cap\Omega_\beta$, both momentum
maps $F_\alpha$ and $F_\beta$ are regular. Contrary to all previous
cases, the focus-focus model $F_\alpha$ is not explicit, and we cannot
simply provide an elementary formula for the transition map
$\phy_{\alpha\beta}$.  However, since $C\setminus\ell$ is simply
connected and a set of regular values of $F_i$, we can invoke the
Liouville\--Mineur\--Arnold action-angle theorem and assert that there
exists a symplectomorphism $ \phy_{i} : F_i^{-1}(C\setminus\ell)
\fleche \T^2\times C'\subset T^*\T^2 = \{(x,\xi) \in \mathbb{T}^2
\times \mathbb{R}^2\} $ such that
$$
F_i=\phy_i^* (h_i(\xi))\,\,\, \textup{for some diffeomorphism}\,\,\,
h_i:C'\fleche C\setminus\ell.
$$ 
Then $\mu_i=\phy_i^*(g_i\circ h_i(\xi))$. Since both $\mu_i$ and $\xi$
are toric momentum maps for the same foliation, there exists a
transformation $H_i\in\textup{Aff}(2,\ZM)$ such that $g_i\circ
h_i=H_i$.

Thus, if $F_\alpha$ is focus-focus and $F_\beta$ is regular, we
introduce the symplectomorphism
\begin{equation}
  \label{equ:ff}
  \phy_{\alpha\beta}:= \tilde{R}_\beta^{-1}\tilde{R}_\alpha
  \tilde{H}_i\circ\phy_i\quad  :
  F_\alpha^{-1}(\Delta_{\alpha\beta})\fleche
  F_\beta^{-1}(\Delta_{\alpha\beta}).
\end{equation}
We verify $F_\beta\circ \phy_{\alpha\beta} =
F_r\circ\tilde{R}_\beta\circ \phy_{\alpha\beta} = R_\alpha H_i
F_r\circ\phy_i = R_\alpha\mu_i = F_\alpha$, so we have shown
(\ref{symplecto:eq}).

We can now include these nodal pieces in the symplectic glueing
construction using Theorem~\ref{theo:glueing}, which defines a
symplectic manifold $M_{A''}$ and a proper map
$$
F_{A''} \colon M_{A''} \to \bigcup_{\alpha \in A''} \Omega_{\alpha}
\subset \mathbb{R}^2.
$$
However $F_{A''}$ is not smooth everywhere, but it is a smooth toric
momentum map outside the preimages of the cuts $\ell_j^+$
($j=1,\dots,m_f$).

\paragraph{Third stage. \emph{Filling in the gaps}.---}
Here we add the open sets $\Omega_\alpha$ that were covering the cuts
$\ell_i$ by switching these lines on the other side.
Let $t_i:=t_{\ell_{\lambda_i}}$ as in Section
\ref{semitoric:sec}.  The cut $\ell_i^+$ is invariant under
$t_i$. The open sets $t_i(\Omega_\alpha)$, $\alpha\in
A\setminus A''$ form a cover of $\ell_i\cap t_i(\Delta)$. Within
the geometry of the new polygon $t_i(\Delta)$, each of these open
sets can be associated with either a regular model, a transversally
elliptic model, or an elliptic-elliptic model (indeed, under the
transformation $t_i$, a fake corner disappears, and a hidden
Delzant corner unhides itself.)

Thus we can add these to our glueing data, which amounts to equip each
such open set $\Omega_\alpha$ with the model $(M_\alpha, \,
t_i^{-1}\circ F_\alpha)$, where $(M_\alpha,\,F_\alpha)$ is
determined as before, but for the transformed polygon
$t_i(\Delta)$.

The transition maps are defined with the same formulas as before,
taking into account that the map $R_\alpha$ is now a piecewise affine
transformation. The cocycle conditions remain valid as well.

Doing this for all indices $i$, because all the $F_{\alpha}$ are
continuous and proper, by Theorem \ref{theo:glueing}, we obtain a
smooth symplectic manifold $M=M_{A}$ equipped with a proper,
continuous map $\mu=F_A$
\begin{eqnarray} \label{mu:map} \mu:M\fleche \bigcup_{\alpha\in
    A}\Omega_\alpha \subset \mathbb{R}^2,
\end{eqnarray}
whose image is precisely $\Delta$.


However, the map $\mu$ is a proper toric momentum map only outside the
cuts $\ell_i$. In other words, $\mu$ fails to be smooth along the cuts
$\ell_i$. (Note that in the symplectic glueing construction, Theorem
\ref{theo:glueing}, we did not make any smoothness assumption on the
$F_{\alpha}$, nor made any conclusion on the smoothness of $F$).

\paragraph{Fourth and final stage. \emph{Recovering smoothness.}---}

In this step we compose the final momentum map $\mu$ in (\ref{mu:map})
on the left by a suitable homeomorphism in order to make it smooth.
Let $\Omega_\alpha$ be the open set containing the node $c_i$. Let
$h_i=g_i^{-1}:D_i\fleche C$. The map $h_i$ is a bilipschitz
homeomorphism fixing the origin and a smooth diffeomorphism outside
the positive vertical axis.  It is of the form
\[
h_i(x,\,y)=(x,\,\eta_i(x,\,y)).
\]
Since $h_i$ is orientation preserving, $\deriv{\eta_i}{y}(x,\, y)>0$
for all $(x,\,y)\in D_i$.  Let $\delta_i>0$ be such that
$[-2\delta_i,\, 2\delta_i]^2\subset D_i$ and consider the vertical
half-strip $\mathcal{S}_{\delta_i}:=[-\delta_i,\,
\delta_i]\times[-\delta_i,\, \infty[$.

\begin{claim} \label{1:claim} There exists a function
  $\tilde{\eta}_i:D_i\fleche C$ such that
  \begin{itemize}
  \item[(1)] $\tilde{\eta}_i(x,\, y)=\eta_i(x,\, y)$ for all
    $(x,\,y)\in D_i \cap \mathcal{S}_{\delta_i}$;
  \item[(2)] $\tilde{\eta}_i(x,\,y)=y$ for all $(x,\,y)\in
    D_i\setminus \mathcal{S}_{2\delta_i}$;
  \item[(3)] $\deriv{\tilde{\eta}_i}{y}(x,\, y)>0$ for all $(x,\,
    y)\in D_i$.
  \end{itemize}
\end{claim}
In order to show this recall that if $f \colon A \to \R$ is smooth and
$A \subset U \subset \mathbb{R}^2$ is closed, then $f$ has a smooth
extension to $\tilde{f}\colon U \to \mathbb{R}$ where $U$ is open, see
for example \cite[Lem.~5.58 and Rmk. below it]{wlokaetal}.  Let us
apply this fact in our situation.  Let $ A_{\delta_i}:=(D_i \cap
S_{\delta_i}) \cup (D_i \setminus \op{Int}(S_{\frac{3\delta_i}{2}})),
$ which is a closed subset of $D_i \subset \mathbb{R}^2$, and let
$\widehat{\eta}_i\colon A_{\delta_i} \to \mathbb{R}$ be the smooth
function given by
\begin{eqnarray} \label{widehateta:map} \widehat{\eta}_i(x,\,y)=
  \left\{ \begin{array}{rl}
      \eta_i(x,\,y) & \textup{ if } (x,\,y)\in \,D_i \cap \mathcal{S}_{\delta_i}; \\
      y & \textup{ if } (x,\,y)\in \, D_i \setminus
      \op{Int}(\mathcal{S}_{\frac{3\delta_i}{2}}) . \end{array}
  \right.
\end{eqnarray}

\begin{figure}[htb]
  \begin{center}
    \includegraphics{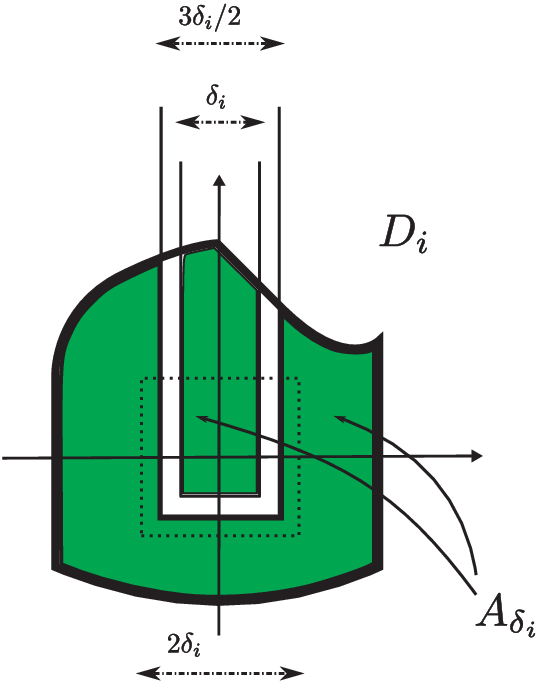}
    \caption{The set $A_{\delta_i}:=(D_i \cap S_{\delta_i}) \cup (D_i
      \setminus \op{Int}(S_{\frac{3\delta_i}{2}}))$, on which
      $\widehat{\eta}_i$ is defined.}
  \end{center}
\end{figure}

Because $A_{\delta_i} \subset D_i$, and $D_i$ is bounded, there exists
a constant $0<c_i<1$ such that $ \deriv{\eta_i}{y}>c $ on
$A_{\delta_i}$ and hence $ \deriv{\widehat{\eta}_i}{y}>c_i
\,\,\textup{on}\,\, A_{\delta_i}.  $ Let $
\zeta_i:=\deriv{\widehat{\eta}_i}{y}-c_i \colon A_{\delta_i} \to
\mathbb{R}, $ which by assumption is strictly positive.  By the above
fact $\zeta_i$ extends to a smooth function $G_i: D_i \to
\mathbb{R}$. Because the proof of the fact preserves non\--negativity,
and $\zeta_i>0$, we have that $G_i \ge 0$.  By possibly shrinking the
size of $D_i$ we can assume that $D_i$ is a disk of radius $r_i>0$
centered at the origin. Let $X_{\delta_i}:=[-r_i,\,
-\frac{3\delta_i}{2}]\cup [\frac{3\delta_i}{2},\, r_i]$,
$Y_{\delta_i}:=[-\delta_i,\,\delta_i]$,
$Z_{\delta_i}:=[-\frac{3\delta_i}{2},\, -\delta_i] \cup [\delta_i,\,
\frac{3\delta_i}{2}]$ and let $\nu^i_1\colon X_{\delta_i} \to
\mathbb{R}$ and $\nu^i_2 \colon Y_{\delta_i} \to \mathbb{R}$ be the
functions given by $\nu^i_1(x):=-\widehat{\eta}_i(x,\,0)$ and
$$\nu^i_2(x):=\widehat{\eta}_i(x,\, -\frac{3\delta_i}{2})
-\int_{0}^{-\frac{3\delta_i}{2}}(G_i(x,\,t)+c_i)\, \op{d}\!t,
$$
where we are using the convention $\int_a^bh=-\int_b^ah$ when $a>b$.
Because $\widehat{\eta}_i$ and $G_i$ are smooth functions, $\nu^i_1$
and $\nu^i_2$ are also smooth. Let $\beta^i \colon [-r_i,\,r_i] \to
\mathbb{R}$ be a smooth extension of the function $X_{\delta_i} \cup
Y_{\delta_i} \to \mathbb{R}$ defined by $\nu^i_1$ on $X_{\delta_i}$
and by $\nu^i_2$ on $Y_{\delta_i}$, which again exists by a partitions
of unity argument.

Consider the function $\widetilde{\eta}_i \colon D_i \to \mathbb{R}$
given by
$$\widetilde{\eta}_i(x,\,y):=
\beta^i(x)+\int_0^y (G_i(x,\,t)+c_i) \, \op{d}\!t.
        $$
        Because $\beta$ is a smooth extension of $\nu^i_1$ and
        $\nu^i_2$, and $G$ is smooth, $\widetilde{\eta}_i$ is smooth.
        We claim that
        $\widetilde{\eta}_i|_{A_{\delta_i}}(x,\,y)=\widehat{\eta}_i(x,\,y)$
        if $(x,\,y) \in A_{\delta_i}$. First assume $x \in
        Y_{\delta_i}$, and moreover that $-r_i \le y \le - \frac{3
          \delta_i}{2}$.  Because $G_i$ is an extension of $g_i$ we
        have that
$$
\widetilde{\eta}_i|_{A_{\delta_i}}(x,\,y) = \nu^i_2(x)+ \Big(
\int_0^{-\frac{3\delta_i}{2}}(G_i(x,\,t)+c_i)\,
\op{d}\!t+\int_{-\frac{3\delta_i}{2}}^y
\deriv{\widehat{\eta}_i}{y}(x,\,t)\, \op{d}\!t\Big),
 $$
 and hence by the fundamental theorem of calculus, and using the
 definition of $\nu^i_2$ we obtain that
 \begin{eqnarray}
   \widetilde{\eta}_i|_{A_{\delta_i}}(x,\,y)=
   \nu^i_2(x)
   +\Big(   \int_0^{-\frac{3\delta_i}{2}}(G_i(x,\,t)+c_i)\, \op{d}\!t
   + (\widehat{\eta}_i(x,\,y)-\widehat{\eta}_i(x,\, -\frac{3\delta_i}{2}))\Big)=\widehat{\eta}_i(x,\,y). 
   \label{es2:eq}
 \end{eqnarray}
 The remaining subcases within the case of $x \in Y_{\delta_i}$ are
 when $-\delta_i \le y \le 0$, which follows by the same reasoning as
 in (a) using the formula for $\nu^i_1$ instead of $\nu^i_2$, the case
 of $0\le y \le r_i$, which is trivial because the extension is
 defined by the original function therein, and the case of $-\frac{3
   \delta_i}{2} \le y \le -\delta_i$, in which $(x,\,y) \notin
 A_{\delta_i}$ so there is nothing to prove. The case of $x\in
 X_{\delta_i}$ follows by the same type of argument as the case of
 $Y_{\delta_i}$.  The case of $x \in Z_{\delta_i}$ is immediate
 because the extension is defined by the original function therein.

 Applying again the fundamental theorem of calculus, because the
 functions $\nu^i_1,\, \nu^i_2,\, \beta^i$ do not depend on $y$, we
 have that
 \begin{eqnarray} \label{eq1} \deriv{\tilde{\eta}_i}{y}=G_i+c_i,
 \end{eqnarray}
 which is strictly positive since $G_i\ge0$ and $c_i>0$.  Because
 (\ref{eq1}) and (\ref{es2:eq}) hold we in turn have, in view of the
 definition (\ref{widehateta:map}) of $\widehat{\eta}$, that
 properties 1, 2, 3 are satisfied.  This concludes the proof of Claim
 \ref{1:claim}

 \vspace{1.7mm}
         
 Let $\Omega_i:=D_i\cup \{(x,\,y) \, | \, y<2\delta_i\}$.  Because of
 the properties 1, 2, 3 of $\tilde{\eta}_i$, the map
$$
\tilde{h}_i:(x,\,y)\mapsto (x,\tilde{\eta}_i(x,\,y))
$$
coincides with $h_i$ in $\mathcal{S}_{\delta_i}$, while it is equal to
the identity outside $\mathcal{S}_{2\delta_i}$. Thus we can extend it
to $\Omega_i$ by letting it to be the identity outside
$D_i\cup\mathcal{S}_{2\delta_i}$. We call this extension
$\tilde{h}_{\Omega_i}$.  Consider the map
$$
\check{h}_{\Omega_i}:=\tilde{h}_{\Omega_i} \circ t_0^{-1},
$$
where $t_0$ is the piecewise affine map $t_\ell$ with $\ell$
being the positive vertical axis. In
$t_0(\Omega\cap\mathcal{S}_{\delta_i})$, it it equal to $h_i\circ
t_0^{-1}$, which is now smooth outside the negative vertical axis
(this follows from~\cite[Thm.~ 3.8]{vungoc}; also from the fact that
it is the homeomorphism that one obtains in the construction of the
generalized momentum map $t_0\circ g_i \circ F_i = t_0\circ
\mu_i$: this amounts to switching the cut downwards.)  Using the claim
at the beginning of this step upside\--down we can modify
$\check{h}_{\Omega_i}$ in $\Omega_i\cap\{y>\delta_i\}$ in such a way
that we can then extend it to be smooth on $t_0(\{y>\delta_i\})$.
We obtain a homeomorphism of $\RM^2$ that we call
$(\check{h}_{\RM^2})_i$.

Define the map $\varphi_i \colon \mathbb{R}^2 \to \mathbb{R}^2$ by
\begin{eqnarray}
  \varphi_i:=R_\alpha\circ (\check{h}_{\RM^2})_i\circ t_0 \circ \nonumber
  R_\alpha^{-1}.
\end{eqnarray}
Because $\varphi_i$ is a composite of homeomorphisms, it is a
homeomorphism.  Moreover, outside of $\mathcal{S}_{2\delta_i}$ we have
that
$$
\varphi_i=R_\alpha\circ (\check{h}_{\RM^2})_i \circ t_0 \circ
R_\alpha^{-1}= R_\alpha\circ (\tilde{h}_{\Omega_i} \circ
t_0^{-1})\circ t_0 \circ R_\alpha^{-1},
$$
and since $\tilde{h}_{\Omega_i}$ is the identity outside of
$\mathcal{S}_{2\delta_i}$ we conclude that $\varphi_i$ is the identity
map outside $\mathcal{S}_{2\delta_i}$.  Now let $\varphi \colon
\mathbb{R}^2 \to \mathbb{R}^2$ be the piecewise defined map
\begin{eqnarray} \label{varphi:map} \varphi(x,\,y):=
  \left\{ \begin{array}{rl}
      \varphi_i(x,\,y) & \textup{ if } (x,\, y)\in \mathcal{S}_{2\delta_i}; \\
      (x,\,y) & \textup{ otherwise }. \end{array} \right.
\end{eqnarray}
Since each $\varphi_i$ is a homeomorphism, and equal to the identity
outside of $\mathcal{S}_{2\delta_i}$, the formula (\ref{varphi:map})
defines a homeomorphism.

\begin{claim} \label{5.2:claim} The map $\tilde{F} \colon M \to
  \mathbb{R}^2$ defined by $\tilde{F}:=\varphi \circ \mu$ is proper,
  and smooth everywhere.
\end{claim}

The properness claim is immediate since $\varphi$ is a homeomorphism
and $\mu$ is proper.

In order to show that $\tilde{F}$ is smooth, consider the map
$\tilde{F}_i \colon M \to \mathbb{R}^2$ defined as a composite
$\tilde{F}_i:= \varphi_i \circ \mu,$
where recall $\mu$ is the map (\ref{mu:map}).  By definition of
$\varphi$, we have that
$\tilde{F}|_{\mathcal{S}_{\delta_i}}=\tilde{F}_i$, and hence to prove
the claim it suffices to show that each $\tilde{F}_i$ is smooth.  To
prove this, we distinguish three cases.
\\
\\
{\emph Case 1}: \emph{in a neighborhood of $c_i$}. In the neighborhood
$\Omega_{\alpha}$ of $c_i$ sent by $R_\alpha^{-1}$ into
$[-\delta,\delta]^2$, we have that
\[
(\check{h}_{\RM^2})_i \, t_0 \, R_\alpha^{-1} =
\check{h}_{\Omega_i}\, t_0 \, R_\alpha^{-1} =
\tilde{h}_{\Omega_i}\, t_0^{-1}\, t_0 \,R_\alpha^{-1} =
h_i\, R_\alpha^{-1}.
\]
Recall that $y_\alpha^*\mu=F_\alpha=R_\alpha\circ \mu_i$.  Therefore
one can write, in the preimage by $\mu$ of this neighbourhood, $
y_\alpha^*(\tilde{F}_i) = y_\alpha^*(h_i\circ \mu_i) = F_i.  $ Since
$F_i$ is smooth, it follows that $\tilde{F}_i$ is smooth in
$\Omega_{\alpha}$.
\\
\\
{\emph Case 2}: \emph{away from the cut $\ell_i$}.  Let $
\Lambda_i:=\bigcup_{j \neq i} \mu^{-1}(\ell_j) \subset \R^2.  $ We
have that
$$
(\check{h}_{\RM^2})_i \,t_0 \,R_\alpha^{-1} = \check{h}_{\Omega}
\,t_0\, R_\alpha^{-1} = \tilde{h}_{\Omega} \, R_\alpha^{-1}
\,\,\, \textup{on the set}\,\,\, (R_\alpha\circ
t_0^{-1})(\{(x,\,y)\, | \, y<-\delta_i/2\}),
$$
which by construction is smooth on this set. Thus
$\tilde{F}_i$ has the same degree of smoothness as $\mu$ on the set
$\mu^{-1}((R_\alpha\circ t_0^{-1})(\{(x,\,y)\, | \,
y<-\delta/2\}))$.
Note that the set $\mu^{-1}\big((R_\alpha\circ
t_0^{-1})(\{(x,\,y) \, | \, y<-\delta_i/2\})\big)$ does not
contain $\mu^{-1}(\ell_i)$.  The same argument applies to the analogue
subsets of $M$ corresponding to the regions $\{(x,\,y) \, | \,
x<-\delta_i/2\}$ and $\{(x,\,y)\, | \, x>\delta_i/2\}$.  On the subset
of $M$ corresponding to the region $\{(x,\, y)\, | \, y>\delta_i/2\}$,
the map $(\check{h}_{\RM^2})_i$ is smooth by construction.  Hence the
map $\tilde{F}_i$ is smooth on $M\setminus \Lambda_i$.
\\
\\
{\emph Case 3}: \emph{along the cut $\ell_i$, away from $c_i$}.
Remark that $t_0R_\alpha^{-1}=R_\alpha^{-1}t_i$.  By
construction of $\mu$ above the open sets $\Omega_\beta$ covering the
cut $\ell_i$, we have that $ y_\beta^*\, \mu =t_i^{-1}\,
F_\beta.  $ Hence
$$
y_\beta^*((\check{h}_{\RM^2})_i \, t_0\, R_\alpha^{-1} \mu) =
y_\beta^*\, ((\check{h}_{\RM^2})_i \, F_\beta)\,\,\, \textup{on the
  set}\,\,\, \mu^{-1}(\Omega_\beta),
$$
and this expression defines a smooth map. Thus $\tilde{F}_i$ is
smooth.
\\
\\
Hence putting cases 1, 2, 3 together we have shown that $\tilde{F}_i$
is smooth on $\mu^{-1}(\Omega_\beta)$ for all $\Omega_\beta$ covering
the cut $\ell_i$, and elsewhere, $\tilde{F}_i$ is as smooth as $\mu$.
This concludes the proof of Claim \ref{5.2:claim}.

\vspace{1.7mm}

Write $\tilde{F}:=(J,\,H)$. We then have the following conclusive
claim.

\begin{claim} \label{5.3:claim} The symplectic manifold $(M,\ \omega)$
  equipped with $J$ and $H$ is a semitoric integrable system.
  Moreover, the list of invariants (i)\--(v) of the semitoric
  integrable system $(M,\, \omega,\, (J,\,H))$ is equal to the list of
  ingredients (i)\--(v) that we started with.  Finally, $M$ is a
  compact manifold if and only $\Delta$ is compact.
\end{claim}

Let us prove this claim. We know from Claim~\ref{5.2:claim} that
$\tilde{F}$ is smooth. Since the first component $J$ is obtained from
glueing proper maps, it follows from Theorem~\ref{theo:glueing} that
$J$ is proper. What's more, the Hamiltonian flow of $J$ is everywhere
periodic of period $2\pi$ because it is true in any local piece
$M_\alpha$. Clearly $\{J,\,H\}=0$, since it is a local property.  It is
also easy to see that the only singularities of $\tilde{F}$ come from
the singularities of the models $F_\alpha$, for the glueing procedure
does not create any additional singularities. Now, near any elliptic
critical value, the homeomorphism $\mu$ is a local diffeomorphism, so
$\tilde{F}$ has the same singularity type as the elliptic model
$F_\alpha$. Finally, near a node we have checked in the proof of
Claim~\ref{5.2:claim} that $\tilde{F}$ is precisely equal to the model
$F_i$, and hence possesses a focus-focus singularity. Thus, provided
we show that $M$ is connected, $(J,\,H)$ is a semitoric system.

Let us now consider its invariants (the connectedness of $M$ will
follow).

\begin{itemize}
\item[(i)] As we mentioned, the singularities of $\tilde{F}$ are only
  elliptic, except for the nodes $c_1,\,\dots,\, c_{m_f}$ above each of
  which we have constructed a focus-focus singularity. Hence we have
  $m_f$ focus-focus singularities.
\item[(ii)] Each focus-focus singularity was constructed by glueing a
  semi-local model with prescribed Taylor series invariant
  $(S_i)^\infty$. Since this Taylor series is precisely a semi-local
  symplectic invariant, it is unchanged in the glued system
  $(M,\tilde{F})$.
\item[(iii)] Thus we have a completely integrable system on $M$ that
  defines an integral affine structure (with boundary) on the image of
  $\tilde{F}$, except at the nodes $c_i$. For any choice of vertical
  half cuts $(\ell_i,\epsilon_i)$, the generalized momentum polygon is
  the image of the affine developing map. But the momentum map $\mu$,
  outside the focus-focus fibres, is precisely such a developing map
  and its image, by the glueing procedure, is the polygon
  $\Delta$. Hence the semitoric polygon invariant of $\tilde{F}$ is
  the orbit of $\Delta_{\scriptop{w}}$. (See Lemma~\ref{lemm:mu}.)

  Notice that this shows that the image of $\mu$ is connected, which
  implies that the total space $M$, obtained by glueing above the
  image of $\mu$, is connected as well.
\item [(iv)] It follows directly from (iii) above and the definition
  of the nodes $c_j$ in~\eqref{equ:cuts} that the volume invariant
  defined in~\eqref{height:eq} is equal to $(h_1,\dots,h_{m_f})$.
\item [(v)] We calculate the twisting indices of our semitoric system
  with respect to the fixed polygon $\Delta$ or, which amounts to the
  same, with respect to the toric momentum map $\mu$. By definition, the
  $j^{\scriptop{th}}$ twist is the integer $\tilde{k}_j$ such that
\[
\DD\mu=T^{\tilde{k}_j}\DD\mu_j,
\]
where $\mu_j$ is the privileged momentum map of the focus-focus
fibration above $c_j$.  From the second stage of the construction, we
know that
\[
\mu = F_\alpha = R_\alpha \circ \mu_j = \tau\circ T^{k_j} \circ \mu_j,
\]
where $\tau$ is some translation. Hence $\DD\mu=T^{k_j}\DD\mu_j$, and
thus $\tilde{k}_j=k_j$.
\end{itemize}

Thus we see that we could prove the second part of the claim because
our construction is by symplectically glueing local pieces with the
appropriate ingredients as in Definition \ref{listofingredients}. This
is an advantage of constructing by glueing local pieces rather than,
for example, a global reduction on a larger space.

This concludes the proof of Claim \ref{5.3:claim}, and hence the proof
of the theorem.

\noindent
\\
Alvaro Pelayo \\
University of California\---Berkeley \\
Mathematics Department,
970 Evans Hall $\#$ 3840 \\
Berkeley, CA 94720-3840, USA.\\
{\em E\--mail}: {apelayo@math.berkeley.edu}

\medskip\noindent

\noindent
V\~u Ng\d oc San\\
Institut de Recherches Math\'ematiques de Rennes\\
Universit\'e de Rennes 1\\
Campus de Beaulieu\\
35042 Rennes cedex (France)\\
{\em E-mail:} {san.vu-ngoc@univ-rennes1.fr}


\begin{thebibliography}{9999}


\bibitem{atiyah}M. Atiyah: Convexity and commuting Hamiltonians.  {\em
    Bull. London Math. Soc.} {\bf 14} (1982) 1--15.


\bibitem{bourbaki} N. Bourbaki: \emph{General Topology}, Elements of
  Mathematics (Chapters I-IV) Springer, 1998

\bibitem{bourbaki2} N. Bourbaki: \emph{Vari\'et\'es diff\'erentielles
    et analytiques}, \'Edition originale publi\'ee par Masson, Paris,
  1967.

\bibitem{brandsma} H. Brandsma: Paracompactness, covers and perfect
  maps, \emph{Topology Explained}, March 2003. Published by
  \emph{Topology Atlas}. Available at:
  http://at.yorku.ca/p/a/c/a/00.htm

\bibitem{daverman} R.J. Daverman: \emph{Decompositions of Manifolds},
  AMS Bookstore 2007.  Also published by Academic Press, Orlando,
  1986.

\bibitem{delzant}T. Delzant: Hamiltoniens p\'eriodiques et image
  convexe de l'application moment. {\em Bull. Soc. Math. France} {\bf
    116} (1988) 315--339.

\bibitem{duistermaat} J.J. Duistermaat.  \newblock On global
  action-angle variables.  \newblock {\em Comm. Pure Appl. Math.},
  33:687--706, 1980.


\bibitem{dufour-molino} J.P. Dufour and P.~Molino.  \newblock
  Compactification d'actions de $\mathbb{R}^n$ et variables
  actions-angles avec singularit{\'e}s.  \newblock In Dazord and
  Weinstein, editors, {\em S{\'e}minaire Sud-Rhodanien de
    G{\'e}om{\'e}trie {\`a} Berkeley}, volume~20, pages
  151--167. MSRI, 1989.



\bibitem{eliasson-these} L.H. Eliasson: 
 Normal forms for Hamiltonian systems with Poisson commuting integrals -- elliptic case, 
 {\it Comment. Math. Helv.} {\bf 65} (1990), no. 1, 4Ð35; and
  {\em Hamiltonian systems with {P}oisson commuting integrals}, 
 PhD thesis, University of Stockholm, 1984.



\bibitem{grs1} M. Gross and B. Siebert: Mirror symmetry via
  logarithmic degeneration data II, arXiv:0709.2290 [math.AG].

\bibitem{grs2} M. Gross and B. Siebert: From real affine geometry to
  complex geometry, arXiv:math/0703822 [math.AG].

\bibitem{grs3} M. Gross and B. Siebert: Mirror symmetry via
  logarithmic degeneration data I, \emph{J. Differential Geom.} 72
  (2006), 169Ð338. [math.AG/0309070].

\bibitem{grs4} M. Gross and B. Siebert: Affine manifolds, log
  structures, and mirror symmetry, \emph{Turkish J. Math.} 27 (2003),
  33Ð60. [math.AG/0211094].

\bibitem{gs}V. Guillemin and S. Sternberg: Convexity properties of the
  moment mapping. {\em Invent. Math.} {\bf 67} (1982) 491--513.



\bibitem{ls}N.C. Leung and M. Symington: Almost toric symplectic
  four-manifolds. arXiv:math.SG/0312165v1, 8 Dec 2003.


\bibitem{miranda-zung} E. Miranda and N. T. Zung: Equivariant normal
  for for non\--degenerate singular orbits of integrable Hamiltonian
  systems. \emph{Ann. Sci.  \'Ecole Norm. Sup.} (4), 37(6):819--839,
  2004.


\bibitem{pelayovungoc} A. Pelayo and S. V\~u Ng\d oc: Semitoric
  integrable systems on symplectic $4$\--manifolds. 
  \emph{Inventiones Math.},
  to appear.



\bibitem{s}M. Symington: Four dimensions from two in symplectic
  topology.  pp. 153--208 in {\em Topology and geometry of manifolds
    (Athens, GA, 2001)}.  Proc. Sympos. Pure Math., 71,
  Amer. Math. Soc., Providence, RI, 2003.



\bibitem{vungoc0} S. V\~u Ng\d oc: On semi\--global invariants for
  focus\--focus singularities.  \emph{Top.} {\bf 42} (2003), no. 2,
  365--380.

\bibitem{vungoc} S. V\~ u Ng\d oc: Moment polytopes for symplectic
  manifolds with monodromy.  {\em Adv. Math.} {\bf 208} (2007), no. 2,
  909--934. 53D20


\bibitem{willard} S. Willard: \emph{General Topology}, Courier Dover
  Publications, 2004.

\bibitem{wlokaetal} J.T. Wloka, B. Rowley and B. Lawruk: {\em Boundary
    Value Problems for Elliptic Systems}, Cambridge University Press
  1995.

\bibitem{zung} N.T. Zung: Symplectic
  topology of integrable hamiltonian systems, I:
  Arnold-Liouville with singularities, {\em Compositio Math.},
  vol. 101, p. 179--215, 1996.

\end{thebibliography}
\end{document}